\documentclass[oneside,reqno]{amsart}
\usepackage{amssymb}
\usepackage{amsmath}
\usepackage{amsthm}
\usepackage{amsbsy}
\usepackage{enumerate}
\usepackage{bm}
\usepackage{hyperref}
\date{\today}
\usepackage{cite}
\usepackage{array}
\usepackage{xcolor}
\usepackage{tikz}
\usepackage{graphics}
\usepackage{subcaption}
\usepackage{mathtools}
\usepackage{relsize}

\hypersetup{
	colorlinks=true, 
	linkcolor=blue,
	filecolor=dviolet,
	urlcolor=blue,}

\newtheorem{localclaim}{Claim}

\makeatletter
\newcommand{\distas}[1]{\mathbin{\overset{#1}{\kern\z@\sim}}}

\newcommand{\distras}[1]{
  \mathbin{\overset{#1}{\kern\z@\resizebox{\wd\mybox}{\ht\mysim}{$\sim$}}}
}

    \newtheorem{theorem}{Theorem}
    \newtheorem{lemma}[theorem]{Lemma}
    \newtheorem{proposition}[theorem]{Proposition}
    \newtheorem{corollary}[theorem]{Corollary}

\theoremstyle{definition} 
    \newtheorem{definition}[theorem]{Definition}
    
    \newtheorem{result}[theorem]{Result}
    \newtheorem{remark}[theorem]{Remark}
    \newtheorem{example}[theorem]{Example}
    \newtheorem{exercise}[theorem]{Exercise}

\def\suchthat{\; : \;}

\newcommand{\E}{\mbox{E}}

\def\Z{\mathbb{Z}}

\def\N{\mathbb{N}}
\def\PP{{\mathcal P}}

\def\R{\mathbb{R}}

\def\Z{\mathbb{Z}}

\def\tends{\rightarrow}

\def\<{\langle}
\def\>{\rangle}

\def\bar{\overline}

\newcommand\Tr{{\mbox{Tr}}}

\newcommand\mnote[1]{} 
\newcommand\be{\begin{equation*}}

\newcommand\ee{\end{equation*}}

\newcommand\ben{\begin{equation}}
\newcommand\een{\end{equation}}
\newcommand\bes{\begin{eqnarray*}}
\newcommand\ees{\end{eqnarray*}}

\newcommand\bex{\begin{exercise}}
\newcommand\eex{\end{exercise}}
\newcommand\beg{\begin{example}}
\newcommand\eeg{\end{example}}
\newcommand\benu{\begin{enumerate}}
\newcommand\eenu{\end{enumerate}}
\newcommand\beit{\begin{itemize}}
\newcommand\eeit{\end{itemize}}
\newcommand\berk{\begin{remark}}
\newcommand\eerk{\end{remark}}
\newcommand\bdefn{\begin{defintion}}
\newcommand\edefn{\end{definition}}
\newcommand\bthm{\begin{theorem}}
\newcommand\ethm{\end{theorem}}
\newcommand\bprf{\begin{proof}}
\newcommand\eprf{\end{proof}}
\newcommand\blem{\begin{lemma}}
\newcommand\elem{\end{lemma}}

\newcommand{\Cov}{\mbox{\rm Cov}}
\newcommand{\sm}{{\raise0.3ex\hbox{$\scriptstyle \setminus$}}}

\def\tends{\rightarrow}

\def\CHI{\mathchoice%
{\raise2pt\hbox{$\chi$}}%
{\raise2pt\hbox{$\chi$}}%
{\raise1.3pt\hbox{$\scriptstyle\chi$}}%
{\raise0.8pt\hbox{$\scriptscriptstyle\chi$}}}
\def\smalloplus{\raise1pt\hbox{$\,\scriptstyle \oplus\;$}}

\usepackage{tikz}

\title [Linear eigenvalue statistics of Hankel matrices] 
{A non-Gaussian limit for linear eigenvalue statistics of Hankel matrices}

\author{Kiran Kumar A.S$^{\ast}$, Shambhu Nath Maurya$^{\dagger}$, Koushik Saha$^{\ddagger}$\\\\
	Department of Mathematics, 
	Indian Institute of Technology Bombay, Mumbai, India}

\date{\today}
\thanks{ $^{\ast}$kiran [at] math.iitb.ac.in, $^{\dagger}$shambhumath4@gmail.com, $^{\ddagger}$koushik.saha [at] iitb.ac.in}  
\begin{document}

\begin{abstract}


This article focuses on linear eigenvalue statistics of Hankel matrices with independent entries. Using the convergence of moments we show that the linear eigenvalue statistics of Hankel matrices for odd degree monomials with degree greater than or equal to three does not converge in distribution to a Gaussian random variable. This result is a departure from the known results, Liu, Sun and Wang (2012), Kumar and Maurya (2022), of linear eigenvalue statistics of Hankel matrices for even degree monomial test functions, where the limits were Gaussian random variables. 
\end{abstract}

\maketitle

\noindent{\bf Keywords:}
 Linear eigenvalue statistics, Hankel matrix, moment method, non-Gaussian variable, central limit theorem.
 \vskip5pt
 \noindent{\bf AMS 2020 subject classification:} 60B20, 60B10, 60F05.


 
\section{Introduction and main results}


The study of linear eigenvalue statistics is a popular area of research in random matrix theory. For an $n \times n$ matrix $A_n$, the linear eigenvalue statistics of $A_n$ is defined as 
\begin{equation} \label{eqn:LES}
\mathcal{A}_n(\phi)= \sum_{i=1}^{n} \phi(\lambda_i),
\end{equation} 
where $\lambda_1 , \lambda_2 , \ldots , \lambda_n$ are the eigenvalues of $A_n$ and $\phi$ is a `nice' test function. The studies on linear eigenvalue statistics started with the study of central limit theorems for linear eigenvalue statistics of Sample covariance matrix \cite{Arharov_Sample_Cov}, \cite{jonsson}.
The results have been obtained for other important classes of random matrices and test functions. Notable among them are the results for polynomial test functions on Wigner matrices by Sinai and Soshnikov \cite{soshnikov1998tracecentral}, on tridiagonal matrices by Popescu \cite{Popescu}, on Toeplitz matrix by Liu, Sun and Wang \cite{liu2012fluctuations} and on circulant matrices by Bose et al. \cite{a_m&s_circulaun2020}.  For results on fluctuations of linear eigenvalue statistics of Wigner and sample covariance matrices, see \cite{girko}, \cite{bai2004clt} and \cite{lytova2009central}.  

In this paper, we study the linear eigenvalue statistics of Hankel matrices for odd degree monomials and polynomials with odd degree terms. 
Hankel matrices are given by $H_n=(x_{i+j-1})_{i,j=1}^n$, where $(x_i)_{i \geq 1}$ is known as the input sequence. Hankel matrices are an important class of patterned matrices and have wide applications both in pure mathematics and other fields of sciences and engineering. In mathematics, Hankel matrices are best known for their connection to the Hamburger moment problem (see \cite{Shohat_Tamarkin}). They also show up in studies of orthogonal polynomials and Pade's approximation \cite{Hankel_pade}, and in error-correcting codes \cite{Jonckheere}. Hankel matrices  have also found applications in areas as diverse as superconductivity \cite{Hankel_fermi_bose}, macro-economics \cite{Masanao}, image processing \cite{Hankel_image_proc}, spectral learning \cite{Hankel_spectral_learn} and spectroscopy \cite{Hankel_spectrospocy}.

Hankel matrices are closely related to another important class of matrices known as Toeplitz matrices and in most cases, their studies go hand in hand. More specifically, for any Toeplitz matrix $T_n$, $P_nT_n$ is a Hankel matrix and conversely for any Hankel matrix $H_n$, $P_n H_n$ is a  Toeplitz matrix, where $P_n = (\delta_{i-1, n-j} )_{i,j=1}^n$ is the backward identity permutation. Since $P_n^{-1} = P_n$, we obtain that any Hankel matrix (and Toeplitz matrix) is of this form. 

For a sequence of random variables $\{x_i\}_{i \in \Z}$, we define the random Toeplitz matrix as $T_n=(x_{i-j})$ and the random Hankel matrix as $H_n=P_nT_n$. In this article, the Hankel matrices considered are always of the form $H_n= P_nT_n$. For a random Hankel matrix $H_n$, we define 
\begin{equation} \label{eqn:wp_Hn_odd}
	w_p := \Tr(A_n^p),
\end{equation}
where $A_n= H_n /\sqrt{n}$. 

To the best of our knowledge, linear eigenvalue statistics of Toeplitz matrices were first studied by Chatterjee in \cite{chatterjee2009fluctuations} . He showed that for test functions $\phi(x)=x^{p_n}$, where $p_n =o(\log n/ \log \log n)$, the linear eigenvalue statistics of symmetric Toeplitz matrices with Gaussian input entries converge to a Gaussian distribution under the \textit{total variation norm}. Later in 2012, Liu et al. \cite{liu2012fluctuations} studied $\mathcal{A}_n(\phi)$ for band Toeplitz and band Hankel matrices with independent input sequence $\{ x_i\}$ obeying the following moment conditions:
\begin{equation}\label{eqn:condition}
\E[x_i]=0, \ \E[x_i^2]=1 \ \forall \ i \in \Z \text{ and} \ \sup _{i \in \Z} \E\big[\left|x_{i}\right|^{k}\big]=\alpha_{k}<\infty  \text { for } k \geq 3.
\end{equation}  
It was shown in \cite{liu2012fluctuations} that for $\phi$ as a monomial and under the normalization $1/\sqrt{n}$, the linear eigenvalue statistics of random Toeplitz matrices converge in distribution to a Gaussian random variable. Additionally, for $\phi (x)= x^p$, where $p$ is an \textbf{even} natural number, the linear eigenvalue statistics of random Hankel matrices too converge in distribution to a Gaussian random variable, under the normalization $1/\sqrt{n}$. In a recent article \cite{Hankel_1}, Kumar and Maurya showed that for odd $p$, the fluctuations are not sensitive to the normalization $1/\sqrt{n}$. The precise statements are as follow.


\begin{result}\label{res:Hankel_band} 
Suppose $H_{n}$ is a random
Hankel matrix with $w_p=\Tr(\frac{H_n}{\sqrt{n}})^p$. 
	\begin{enumerate}
		\item[(a)] \textbf{(Theorem 6.4, \cite{liu2012fluctuations})} 
If the entries of $H_n$ satisfies (\ref{eqn:condition}) and $\E [x_{i}^{4}] = \kappa  \ \ \forall \ i \in \Z$, then for even $p\geq 2$, as $n\rightarrow \infty$,
  $$\displaystyle \frac{1}{\sqrt{n}} \bigl\{ w_p - \E[w_p]\bigr\} \stackrel{d}{\rightarrow}  N(0,\sigma_p^2),$$
  	where $ N(0,\sigma_p^2)$ is a Gaussian random variable with an appropriate covariance structure $\sigma_p^2$.
  	\vskip3pt
		\item[(b)] \textbf{(Theorem 3, \cite{Hankel_1})} If the entries of $H_n$ satisfies (\ref{eqn:condition}), then for odd $p\geq 1$,  as $n\rightarrow \infty$,  
		$$\displaystyle \frac{1}{\sqrt{n}}  \bigl\{ w_p - \E[w_p]\bigr\} \stackrel{d}{\rightarrow}  0,$$
	\end{enumerate}
\end{result}

Our main result provides the limiting behaviour of linear eigenvalue statistics of Hankel matrices for odd degree monomial test functions.
\begin{theorem}\label{thm:moments}
	Let $H_n$ be a random Hankel matrix with an input sequence obeying (\ref{eqn:condition}) and $w_p=\Tr(\frac{H_n}{\sqrt{n}})^p$. Then for every odd $p \geq 3$ and $k \in \N$, the limit of $k$-th moment of $w_p$ are given by
	\begin{equation} \label{eqn:lim_moments_wp}
		\beta_k := \lim_{n \rightarrow \infty} \E[w_p^k]=
		\begin{cases}
			\displaystyle \sum_{ \pi \in \mathcal{P}_2(pk)} \frac{1}{2^{m(\pi)}}f_k(\pi) & \quad \mbox{if } k \text{ is even},\\
			0 & \quad \text{if } k \text { is odd},
		\end{cases}
	\end{equation}
	where $f_k(\pi)$ and $m(\pi)$ are as given in Definition \ref{defn: partition integrals} and (\ref{eqn:m pi}), respectively. 
	
	Furthermore for each odd $p\geq 3$, there exist probability measures $\Gamma_p$ on $\mathbb{R}$  with moment sequence $\{\beta_k \}$ and any such $\Gamma_{p}$ has a non-Gaussian distribution with unbounded support.
\end{theorem}
\begin{remark}
(i)	For $p=1$, 
	\begin{align*}
		w_1= \frac{1}{\sqrt{n}} \Tr(H_n) = \frac{1}{\sqrt{n}} \sum_{i=-(n-1) \atop i \ odd}^{(n-1)} x_{i}.
	\end{align*}
	It follows from \textit{central limit theorem} that $w_1$ converges in distribution to a Gaussian random variable.
	\vskip3pt
 \noindent (ii)  In Section \ref{subsec:exi+uniq_Gammap}, we show that $\{\beta_{k}\}$ does not obey Carleman's condition and therefore $\Gamma_{p}$ might not be unique. Regardless if we assume that $\Gamma_{p}$ is the unique distribution with moment sequence $\{\beta_k\}$, then by Theorem \ref{thm:moments} and moment method, $w_{p} \stackrel{d}{\rightarrow}  \Gamma_p$ for any choice of input sequence obeying (\ref{eqn:condition}). In spite of whether $\Gamma_{p}$ is unique or not, for each $p \geq 3$, $w_p$ does not converge in distribution to a Gaussian random variable. For more details, see Proposition \ref{pro:not_normal_Hn} and Corollary \ref{cor:non-Gaussian_limit}. 
 	\vskip3pt
  \noindent (iii) If $w_p$ converges in distribution, then the moment sequence of the limiting variable will be $\{\beta_k\}$ and the corresponding density will be symmetric which can also be seen in Figure \ref{fig:Limiting_dist}.
  	\vskip3pt
  \noindent (iv) We also show that for a real polynomial $Q(x)$ with  odd degree terms only,	 
 	\begin{equation} \label{thm:wQ_ind}
 		\E[\Tr\{Q(A_n)\}]^k \to \tilde{\beta}_k, \mbox{ as $n\rightarrow	\infty$,}
 	\end{equation}
 	where  $\tilde{\beta}_k$ is as given in (\ref{eqn:lim_moments_wQ}). If there exists a unique distribution, say $ \Gamma_Q$, then $\Tr\{Q(A_n)\} \stackrel{d}{\rightarrow} \Gamma_Q$.
 		\vskip3pt
 	\noindent (v)
 	Our results are in agreement with the simulations in Figure \ref{fig:Limiting_dist}. Simulations further suggest that the linear eigenvalue statistics converge to a universal non-Gaussian distribution which is \textit{unimodular} and \textit{absolutely continuous}.
\end{remark}



In \cite{liu2012fluctuations}, the fluctuation of linear eigenvalue statistics of Toeplitz matrix $T_n$ was studied using a trace formula of the following form
\begin{align*} 
	\Tr(T_n)^p =
		\displaystyle \sum_{i=1}^{n} \sum_{j_{1}, \ldots, j_{p}=-n}^{n} \prod_{r=1}^{p} a_{j_{r}} \prod_{\ell=1}^{p} \chi_{[1, n]} \left(i-\sum_{q=1}^{\ell}(-1)^{q} j_{q}\right) \delta_{0} (\sum_{q=1}^{p} j_{q}), 
\end{align*}
where $\delta$ is the Dirac function and $\chi$ is the indicator function.
 They also derived a closed form of trace formula for Hankel matrices and established the fluctuations of linear eigenvalue statistics of Hankel matrices when the test function is an even degree monomial. In both the situations, the Dirac function appearing in the trace formulas does not depend on ``$i$". 
 
   Now in this article, we study the fluctuations of linear eigenvalue statistics of Hankel matrices when the test function is an odd degree monomial. We use a closed form of trace formula (see Result \ref{res:trace_Hn}) for Hankel matrix to find the limiting moment sequence of  $\Tr(\frac{H_n}{\sqrt{n}})^p$. Note from Result \ref{res:trace_Hn} that for odd $p$, the Dirac function associated in the trace formula $\Tr(\frac{H_n}{\sqrt{n}})^p$ depends on ``$i$". The argument in \cite{liu2012fluctuations} will not work for the study of  linear eigenvalue statistics of Hankel matrices when the test function is an odd degree monomial. For this case, we built a nice connection between the trace formula and a specific type of signed graph to find out the limiting moment sequence. We use some combinatorial arguments to show that the limiting variable is non-Gaussian.

 \begin{figure}[h]
	\centering
	\includegraphics[height=100mm, width =150mm ]{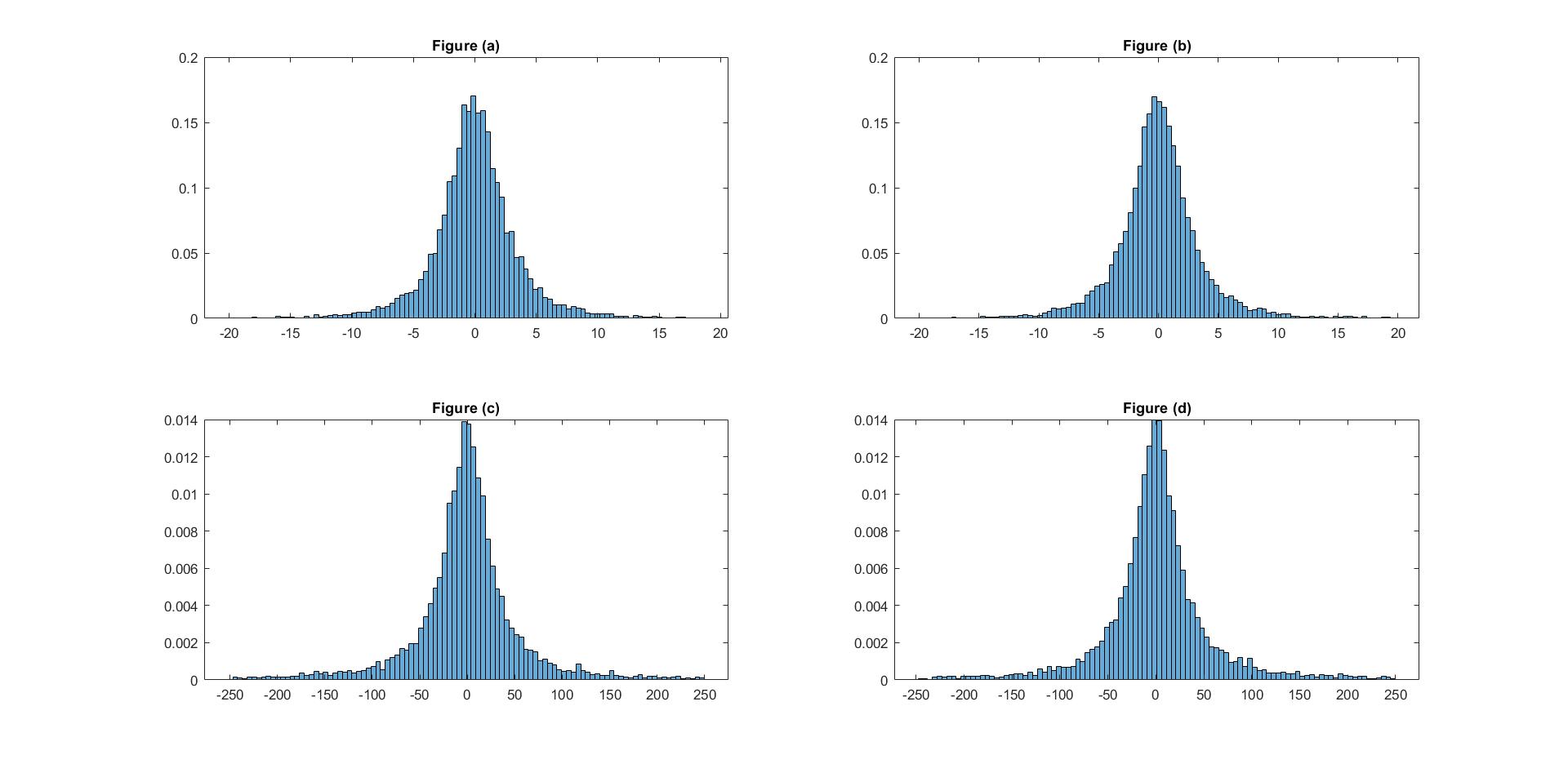}
	\caption{ (a) The histogram of $w_3$ with i.i.d. centred and normalized $U[0,1]$ as the input sequence. (b) The histogram of $w_3$ with i.i.d. standard Gaussian distribution as the input sequence. (c) The histogram of $w_7$ with i.i.d. centered and normalized $U[0,1]$ as the input sequence. (d) The histogram of $w_7$ with i.i.d. standard Gaussian distribution as the input sequence. In all cases, the size of the Hankel matrix is $200 \times 200$ and a total of 10000 random matrices are taken.}\label{fig:Limiting_dist}
\end{figure}

Now we briefly outline the rest of the manuscript. In Section \ref{sec:pre_Hankelodd} we introduce some combinatorial objects and the results associated with them, needed for the proofs of theorems. In Section \ref{sec:wp_ind} we prove Theorem \ref{thm:moments} and discuss the existence and uniqueness of the measure corresponding to the limiting moment sequence. In Section \ref{sec:prop_Gammap}, we study some properties of the limiting distribution of the linear eigenvalue statistics of Hankel matrices.
\section{Preliminaries} \label{sec:pre_Hankelodd}

We first introduce certain partitions and some integrals associated with them. Later, we introduce the concept of matching and the trace formula for Hankel matrices. 
Towards the end of the section, we introduce signed graphs and prove a result on particular types of labelling, named as all-odd labelling and all-even labelling.

\begin{definition}
   Consider the set $[n]=\{ 1,2 ,\ldots , n\}$. A partition $\pi$ of $[n]$ is called a \textit{pair-partition} if each block of $\pi$ has exactly two elements. If $i,j$ belong to the same block of $\pi$, we write $i \sim_{\pi} j$. The set of all pair-partitions of $[n]$ is denoted by $\PP_2(n)$. Clearly, $\PP_2(n)= \emptyset$ for odd $n$. 
\end{definition}

For a partition $\pi$ of $[n]$, we define a canonical ordering of its blocks by arranging the blocks in the increasing order of their smallest elements. Suppose the partition $\pi$ contains $k$ blocks $B_1, B_2,\ldots, B_k$ arranged in the increasing order, then the ordering gives a surjective function, also denoted by $\pi: [n ]\rightarrow [k]$, given by 
\begin{equation}\label{eqn: pi map}
\pi (i)=j, \  \ \mbox{if $i \in B_j$}.
\end{equation}

From a partition $\pi$ and a family of random variables $\{ y_i \}_{i \in [k]}$, we consider a family of random variables $\{ z_i \}_{i \in [n]}$ defined by $z_i=y_{\pi(i)}$. Using this, we introduce the following integrals:

\begin{definition}\label{defn: partition integrals}
	Let $p,k$ be natural numbers such that $kp$ is even and $\pi \in \PP_2(kp)$. Also for $y_1, y_2, \ldots , y_{\frac{pk}{2}}$ and $1 \leq r \leq k$, let $x_r = \frac{1}{2}\left(\sum_{q=1}^{p}(-1)^q y_{\pi((r-1) p+q)}+1 \right)$. Now for $1 \leq r \leq k$, we define
\begin{enumerate} [(i)]
	\item $U_{r}(p) =\prod_{l=1}^{p} \chi_{[0,1]}\left(x_{r}-\sum_{q=1}^{l}(-1)^{q} y_{\pi((r-1) p+q)}\right)$  and
	\item $ f_k(\pi) = \displaystyle  \int_{[-1,1]^{\frac{pk}{2}}} \prod_{r=1}^k U_{r}(p)  \, \mathrm{d}y_1 \mathrm{d}y_2\cdots \mathrm{d}y_{pk/2}$,
\end{enumerate} 
where $\chi_A$ denotes the indicator function of the set $A$.
\end{definition}
Even though, $f_k(\pi)$ depends on $p$, we are avoiding it to lighten the notations.


\begin{definition}\label{defn:cross-matched int}
Let $p_1,p_2$ be natural numbers such that $p_1+p_2$ is even and $\pi \in \PP_2(p+q)$. For $x_1= \frac{1}{2}\left(\sum_{q=1}^{p_1}(-1)^q y_{\pi (q)}+1 \right)$ and $x_2= \frac{1}{2}\left(\sum_{q=1}^{p_2}(-1)^q y_{\pi (p_1+q)}+1 \right)$, we define  
\begin{enumerate}[(i)]
\item $\tilde{U}_{p_1}=\prod_{\ell=1}^{p_1} \chi_{[0,1]}\left(x_{1}-\sum_{q=1}^{\ell}(-1)^{q} y_{\pi(q)}\right)$,
\item $\tilde{U}_{p_2}=\prod_{\ell=1}^{p_2} \chi_{[0,1]}\left(x_{2}-\sum_{q=1}^{\ell}(-1)^{q} y_{\pi(p_1+q)}\right)$ and
\item $g_{p_1,p_2}(\pi)=\displaystyle \int_{[-1,1]^{\frac{p_1+p_2}{2}}} \tilde{U}_{p_1} \tilde{U}_{p_2} \, \mathrm{d}y_1 \mathrm{d}y_2 \cdots \mathrm{d}y_{\frac{p_1+p_2}{2}}$.
\end{enumerate}
\end{definition}

For a vector $J= (j_1,j_2, \ldots, j_p) \in \Z^p$, we define the multi-set $S_J$ as 
\begin{equation}\label{def:S_J}
S_J=\{ j_1, j_2, \ldots , j_p \}.
\end{equation}
For a sequence of vectors $J_1, J_2, \ldots $, we shall use the notation $J_r=(j_1^r, j_2^r,\ldots , j_p^r)$ to denote the components of $J_r$ and $S_{J_r}$ to denote the multi-set associated with $J_r$.

The concept of matching is an important combinatorial notion connected to random matrices. Here, we define the following two notions of matching connected to Hankel matrices.
\begin{definition}
 \begin{enumerate}[(i)]
 \item For $J \in \mathbb{Z}^p$, $j \in \Z$ is said to be a \textit{self-matched element} if $j$ appears at least twice in $S_J$.
 \item  For $J_1 \in \mathbb{Z}^p$ and $J_2 \in \mathbb{Z}^q$, $j \in \Z$ is said to be a \textit{cross-matched element} if $j \in S_{J_1} \cap S_{J_2}$. Additionally, we say two vectors $J_1$ and $J_2$ are \textit{cross-matched} if $S_{J_1} \cap S_{J_2}$ is non-empty.
 \end{enumerate}
\end{definition}
 The following definition provides the natural extension of the concept of matching to any partition $\pi$ of $[p+q]$.
 \begin{definition} \label{def:par_self+crossmatch}
 	Let $\pi$ be a partition of $[p+q]$ for fixed $p,q \in \N$. Consider $J_1=(1,2,\ldots,p)$ and $J_2=(p+1,p+2,\ldots,p+q)$. Then,
 	\begin{enumerate}[(i)]
 		\item  we say an element $j \in J_i$ is self-matched if the intersection of the block of $\pi$ containing $j$, and $S_{J_i}$ has cardinality at least two. 
 		\item  we say an element $j, 1 \leq j \leq p+q$ is  cross-matched if the block of $\pi$ containing $j$, has a non-empty intersection with both $S_{J_1}$ and $S_{J_2}$.
 	\end{enumerate}
 We say a partition $\pi$ is cross-matched if at least one block of $\pi$ has non-empty intersection with both $S_{J_1}$ and $S_{J_2}$. Note that if $J_1$ and $J_2$ are as above and $(p+q)$ is odd, then every $\pi \in \PP_2(p+q)$ is cross-matched.	
 \end{definition}

A trace formula for the product of band Hankel matrices was stated in \cite{Hankel_1}. Now in the following result, we recall the trace formula for the product of Hankel matrices by choosing band width equal to the order of matrices. 

\begin{result}  [Result 4, \cite{Hankel_1}] \label{res:trace_Hn} 
Suppose  $H^{(r)}_{n}$ are Hankel matrices with input sequence $\{ x^{(r)}_i\}_{ i \in \mathbb{Z} }$ for $r=1,2, \ldots$, respectively. Then
\begin{align} \label{def:trace_Hn}
&\Tr(H^{(1)}_{n} H^{(2)}_{n} \cdots H^{(p)}_{n})  \nonumber \\
 & \quad = \begin{cases} 
    \displaystyle \sum_{i=1}^{n} \sum_{j_{1}, \ldots, j_{p}=-n}^{n} \prod_{r=1}^{p} x^{(r)}_{j_{r}} \prod_{\ell=1}^{p} \chi_{[1, n]} \left(i-\sum_{q=1}^{\ell}(-1)^{q} j_{q}\right) \delta_{0} (\sum_{q=1}^{p}(-1)^{q} j_{q}), &p \; \text{even}; \\
     \displaystyle  \sum_{i=1}^{n} \sum_{j_{1}, \ldots, j_{p}=-n}^{n} \prod_{r=1}^{p} x^{(r)}_{j_{r}} \prod_{\ell=1}^{p} \chi_{[1, n]}\left(i-\sum_{q=1}^{\ell}(-1)^{q} j_{q}\right) \delta_{2 i-1-n} (\sum_{q=1}^{p}(-1)^{q} j_{q}), & p \; \text{odd}, 
   \end{cases}
\end{align}
where $\delta_x$ is the Dirac delta function at $x$ and $\chi$ is the indicator function.
\end{result}

Our next definition is connected to the trace formula (\ref{def:trace_Hn}).  
For an odd number $p$ and $i$ such that $1 \leq i \leq n$, we define 
\begin{align} \label{eqn:Ap_HnOdd}
    A_{p,i} &=\Big\{(j_{1}, j_2, \ldots, j_{p}) \in\left\{0, \pm 1, \ldots, \pm n\right\}^{p}: \sum_{q=1}^{p} (-1)^{q} j_{q}=2i-1-n\Big\}, \nonumber \\
    A_p &= \bigcup_{i=1}^{n}A_{p,i}. 
\end{align}

\begin{definition}\label{def:B_p_1,p_2}
Let $p_1, p_2 \ldots , p_r$ be finitely many odd natural numbers. We define $B_{p_1, p_2, \ldots , p_r} \subseteq A_{p_1} \times A_{p_2} \times \cdots \times A_{p_r}$ as the set of all $(J_1, J_2, \ldots , J_r) \in A_{p_1} \times A_{p_2} \times \cdots \times A_{p_r} $, such that each element of the multi-set $\bigcup_{\ell =1}^r S_{J_\ell}$ has cardinality at least two, where $A_{p_i}$ is as defined in (\ref{eqn:Ap_HnOdd}). When $r$ is clear from context and $p_1=p_2=\cdots = p_r =p$, we denote $B_{p_1, p_2, \ldots , p_r}$ simply by $B_{p}$. 
\end{definition}
For real-valued functions $f$ and $g$, we say $f=O(g)$ if there exists a  constant $C>0$ and $x_0 \in \R$ such that $|f(x)| \leq Cg(x)$ for all $x \geq x_0$. The next lemma gives the order of cardinality of $B_{p_1 , p_2, \ldots ,p_r}$.
\begin{lemma}\label{lem:card B_p}
For odd natural numbers $p_1 , p_2, \ldots, p_r$, the cardinality of $B_{p_1, p_2, \ldots ,p_r}$ is given by $$\#B_{p_1, p_2, \ldots ,p_r}= O(n^{\lfloor\frac{p_1+p_2+\cdots + p_r}{2}\rfloor}),$$
where $\#\{\cdot\}$ denotes the cardinality of the set $\{\cdot\}$ and $\lfloor x \rfloor$ denotes the greatest integer less than or equal to $x$.
\end{lemma}
\begin{proof}
Consider a vector $(J_1, J_2, \ldots , J_r) \in B_{p_1, p_2, \ldots ,p_r}$ with $J_\ell = (j_1 ^\ell, j_2^\ell,\ldots,j_{p_\ell}^\ell)$ for each $1 \leq \ell \leq r$. Our objective here is to enumerate the number of possibilities for vectors $(J_1, J_2, \ldots , J_r) \in B_{p_1, p_2, \ldots ,p_r}$. 

Define a relation $\sim$ on $\{(\ell,s) : 1\leq \ell \leq r, 1\leq s \leq p_\ell \}$ by $(\ell_1,s_1) \sim (\ell_2,s_2)$ if $j_{s_1}^{\ell_1}=j_{s_2}^{\ell_2}$. Observe that the number of possible partitions $\pi$ of $\{(\ell,s) : 1\leq \ell \leq r, 1\leq s \leq p_\ell \}$ is finite. We prove that for each partition $\pi$, the number of choices of $(J_1, J_2, \ldots , J_r)\in B_{p_1, p_2, \ldots ,p_r}$ is at most $O(n^{\lfloor\frac{p_1+p_2+\cdots + p_r}{2}\rfloor})$. For that, fix a particular partition $\pi$ and define for each $1 \leq \ell \leq r$, $T_\ell$ as the cardinality of the set $\{ j_s^\ell \in S_{J_\ell}: j_s^\ell \not\in S_{J_u}, \, \forall \, u < \ell \}$, where $S_{J_\ell}$ and $S_{J_u}$ are as defined in Definition \ref{def:S_J}. We enumerate the vectors $(J_1, J_2, \ldots , J_r)$ by sequentially allotting values for components of each $J_\ell$.

Notice that the degree of freedom of choosing elements from $A_p$ equals to the degree of freedom of choosing $j_k$'s freely as once $j_k$'s are fixed, $i$ becomes fixed accordingly. Therefore, the number of choices for choosing components of $J_1$ is $O(n^{T_1})$. Continuing forward in the same fashion, it is clear that once $J_1, J_2, \ldots , J_{\ell -1}$ are chosen, the number of ways of choosing $J_\ell \in A_{p_\ell}$ is $O(n^{T_\ell})$. Thus the total number of possibilities is $O(n^{T_1 + T_2 + \cdots +T_r})$. 

Observe that $T_1 + T_2 + \cdots +T_r$ is the number of distinct values for components of the vectors $J_\ell , 1 \leq \ell \leq r$, and because each component is repeated at least twice, it follows that $T_1 + T_2 + \cdots +T_r$ is bounded above by $\lfloor\frac{p_1+p_2+\cdots + p_r}{2} \rfloor$. This completes our proof.
\end{proof}
Recall that in Definition \ref{def:B_p_1,p_2}, we had imposed the condition that each element of the multi-set $\bigcup_{\ell=1}^{r} S_{J_\ell}$ should be repeated at least twice. Note that, this condition plays a role only in obtaining an upper bound for $T_1 + T_2 + \cdots +T_r$. Therefore the cardinality of any subset of  $U \subseteq A_{p_1} \times A_{p_2} \times \cdots \times A_{p_r}$ is at most $O(n^{T})$, where $T$ is the maximum of $T_1 + T_2 + \cdots +T_r$ among elements of $U$.
We summarize this in the following corollary.
\begin{corollary}\label{cor:cardinality Uk}
Let $p_1 , p_2, \ldots, p_r$ be finitely many odd natural numbers and $T_1, T_2, \ldots, T_r$ as defined in the proof of Lemma \ref{lem:card B_p}. Let $U_k  \subseteq A_{p_1} \times A_{p_2} \times \cdots \times A_{p_r}$ be the collection of all elements $(J_1,J_2,\ldots , J_r)$ such that $T:= T_1 + T_2 +\cdots + T_r$ is less than or equal to $k$, where $A_{p_i}$ is as defined in (\ref{eqn:Ap_HnOdd}). Then the cardinality of $U_k$ is given by
$$\#U_k= O(n^k).$$
\end{corollary}
Now  we introduce certain concepts from graph theory which will be used later in the proofs of theorems.
\begin{definition} \label{def:odd_even_labelling}
A \textit{signed graph} is a graph $G=(V,E)$ with a labelling of edges $f:E \rightarrow\{+1, -1\}$. A signed graph induces a canonical labelling $f_*: V \rightarrow \{+1, -1\}$ on vertices, defined by $f_*(v)=\prod_{u \in V: (u,v)\in E} f(u,v)$.

A labelling $f$ of edges such that the canonical labelling on all vertices are $-1$ is called an \textit{all-odd labelling} and a labelling such that the canonical labelling on all vertices are +1 is called an \textit{all-even labelling}. For a given graph $G$, we denote all-odd labellings and all-even labellings of $G$ by $G^{(ao)}$ and $G^{(ae)}$, respectively. 
\end{definition}
\begin{lemma}\label{lemma:cyclomatic number}
Let $G=(V,E)$ be a finite connected graph with no loops and such that $\#V$ is even. Then 
\begin{equation*}
 \#G^{(ao)}=  \# G^{(ae)} = 2^{\#E-\#V+1}.
\end{equation*}
\end{lemma}
\begin{proof}
We prove this lemma only for the all-odd labelling. A similar argument will  work for the proof of the all-even labelling.

First we prove the lemma for the special case when $G$ is a tree with $\# V$ as even. We use induction on $k$, where $2k= \# V$.
For a tree, we have $\# E = \# V -1$ and therefore our aim is to show that for all trees, there exists a unique all-odd labelling. For $k=1$, the proof is trivial, as there exists only one edge. Suppose the result holds for some $k \geq 1$. Let $G$ be a tree such that $\# V =2k+2$. A vertex $v \in V$, is called a \textit{leaf} if $\operatorname{deg}_G(v)=1$. For a finite tree $G$, one of the following cases always occur.
\begin{enumerate}[(i)]
\item There exists a vertex $v \in V$ and $v_1, v_2 \in V$ such that $(v,v_1), (v,v_2) \in E$ and $v_1, v_2$ are leaves,
\item there exist $v, v_1 \in V$ such that $(v,v_1) \in E, \operatorname{deg}_G(v)=2$ and $v_1$ is a leaf. 
\end{enumerate}
Suppose $G$ obeys (i). Consider the graph $G^{\prime}$ obtained from $G$ by removing $v_1,v_2$ and the edges connected to them. By the induction hypothesis, there exists a unique all-odd labelling $f^{\prime}$ on $G^{\prime}$. We construct a labelling on $G$ by defining $f(u,w)= f^{\prime}(u,w)$ for all $(u,w) \in E(G^{\prime})$, $f(v,v_1)=-1$ and $f(v,v_2)=-1$. This ensures that $f$ is an all-odd labelling. 

Now, suppose $G$ obeys (ii). Consider the graph $G^{\prime}$ obtained from $G$ by removing $v,v_1$ and the edges connected to them. Again by the induction hypothesis, there exists a unique all-odd labelling $f$ on $G^{\prime}$. We define $f(u,w)= f^{\prime}(u,w)$ for all $(u,w) \in E(G^{\prime})$, $f(v,v_1)=-1$ and $f(v,w)=+1$, where $(v,w) \in E(G)$. 

To prove that the all-odd labelling on $G$ is unique, notice that given an all-odd labelling $f^\prime$ on $G^\prime$ obtained in one of the above ways, there exists a unique labelling $f$ on $G$ such that $f\big|_{E(G^\prime)}= f^\prime$. Suppose $f_1$ and $f_2$ are distinct all-odd labellings on $G$. Then $f_1\big|_{E(G^\prime)}$ and $f_2\big|_{E(G^\prime)}$ are all-odd labellings on $G^\prime$. By the induction hypothesis, $f_1\big|_{E(G^\prime)}=f_2\big|_{E(G^\prime)}$, which implies that $f_1=f_2$. 

Now we prove the general case, using induction on $n= (\#E - \#V)$. The base case here is $n=-1$, below which the graph $G$ ceases to be connected. For $n=-1$, $G$ is always a tree and therefore, the result holds. 

Suppose the result holds for $n=k \geq -1$ and $\# V$ as even. Let $G$ be a graph such that $\#E - \#V= k+1$. Since $k+1 \geq 0$, there exists at least one cycle in $G$.
Let $(v_1, v_2)$ be an edge on the cycle.  Fix a path $P$ from $v_1$ to $v_2$, other than the edge $(v_1,v_2)$. Consider the connected graph $G^{\prime}$ obtained from $G$ by removing the edge $(v_1, v_2)$. Let $f^{\prime}$ be an all-odd labelling on $G^{\prime}$. For the rest of the proof, we fix this particular choice of path $P$. 

We define two labellings $f_1$ and $f_2$ on $G$ as
\begin{align*}
	f_1(u,v)=
	\begin{cases}
	f^{\prime}(u,v) &  \mbox{if } \	(u,v) \in E(G^{\prime}) \setminus P \\
		f^{\prime}(u,v) &  \mbox{if } \		(u,v) \in  P \\
		+1 &  \mbox{if } \	 u=v_1, v=v_2,
	\end{cases}
\end{align*}
and 
\begin{align*}
	f_2(u,v)=
	\begin{cases}
		f^{\prime}(u,v) &  \mbox{if } \		(u,v) \in E(G^{\prime}) \setminus P \\
		- f^{\prime}(u,v) &  \mbox{if } \		(u,v) \in  P \\
		-1 &  \mbox{if } \	 u=v_1, v=v_2.
	\end{cases}
\end{align*}
Then it follows that $f_1$ and $f_2$ are all-odd labellings on $G^\prime$. Furthermore, $f_1$ and $f_2$ are the only possible all-odd labellings on $G$ such that they agree with $f^\prime$ on $ E(G^{\prime}) \setminus P$. Thus we get
$$\# G^{(ao)} \geq 2 \times \#   G^{\prime (ao)}.$$ 
To prove the reverse inequality, we need to show that any all-odd labelling $f$ on $G$ can be obtained from an all-odd labelling $f^\prime$ on $G^{\prime}$ such that $f$ and $f^\prime$ are identical on $E(G^{\prime}) \setminus P$. 
	For $f$ such that $f(v_1,v_2)=+1$, the appropriate $f^\prime$ is \begin{align*}
		f^\prime(u,v)=
		\begin{cases}
			f(u,v) &  \mbox{if } \		(u,v) \in E(G^{\prime}) \setminus P \\
		      f(u,v) &  \mbox{if } \		(u,v) \in  P, \\
		\end{cases}
	\end{align*}
	 and for $f$ such that $f(v_1,v_2)=-1$, the appropriate $f^\prime$ is\begin{align*}
f^\prime(u,v)=
	\begin{cases}
		f(u,v) &  \mbox{if } \		(u,v) \in E(G^{\prime}) \setminus P \\
		-f(u,v) &  \mbox{if } \		(u,v) \in  P. \\
	\end{cases}
\end{align*}

Thus we get that,
\begin{align*}
\# G^{(ao)} &= 2 \times \#  G^{\prime (ao)}
= 2 \times 2^{k-1} = 2^k = 2^{\# E -\# V +1}.
\end{align*} 
This completes the proof of lemma.
\end{proof}
 \begin{figure}[h]
	\centering
	\includegraphics[height=80mm, width =150mm ]{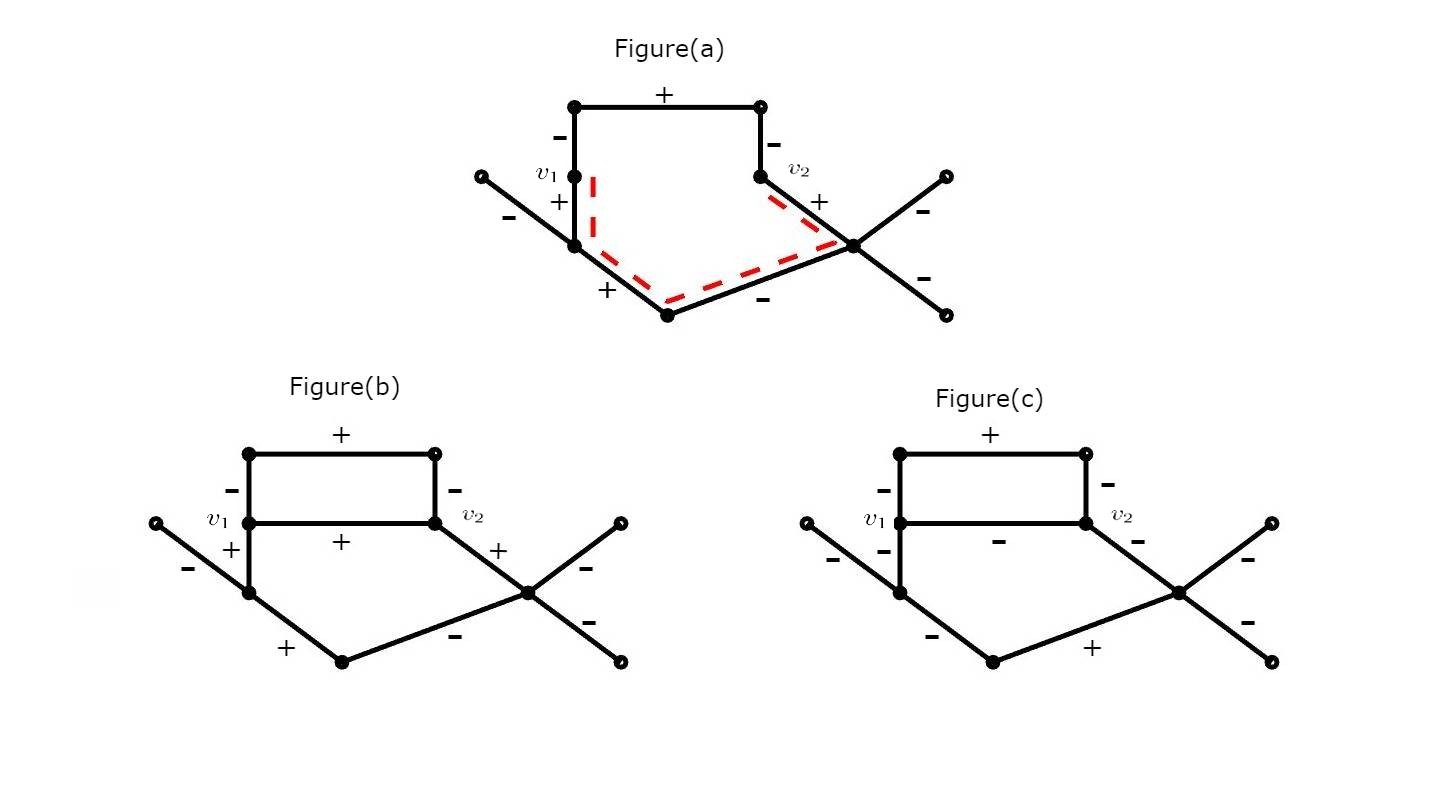}
	\caption{ An example of constructing all-odd labellings from an all-odd labelling of a subgraph. (a) is a sub-graph of the graphs (b) and (c), with the dotted line representing a path from $v_1$ to $v_2$. (b) and (c) are all-odd labellings obtained from an all-odd labelling on $(a)$ with $P$ as the path given by the dotted line.} \label{fig:all-odd labelling}
\end{figure}
The number $(\# E -\# V +1)$ is known as the \textit{cyclomatic number} or the \textit{first Betti number} of a connected graph. 

 Now we define a graph associated with pair-partitions. Let $\pi \in \mathcal{P}_2(pk)$ be a pair-partition of $[pk]$.
We  construct a graph $G_{\pi}=(V_{\pi},E_{\pi})$ associated with $\pi$ in the following fashion: We define $V_{\pi}=\{ 1, 2,\ldots , k\}$ and for $r \neq s$, $(r,s) \in E_{\pi}$ if $\pi$ contains a block $B$ such that $B=\{ (r-1)p+q_1, (s-1)p +q_2 \}$ for some $0 \leq q_1,q_2<p$. A maximal connected subgraph of $G_\pi$ is called a cluster.

For a pair-partition $\pi \in \mathcal{P}_2(pk)$ and integers $j_1, j_2, \ldots, j_{pk/2}$, the concept of  connectedness and cluster can be extended to vectors  $(J_1,J_2,\ldots , J_k)$ where $J_r = (j_1^r,j_2^r,\ldots , j_p^r)$ with $j_q^r= j_{\pi\left((r-1)p+q\right)}$ for all $r,q$ and $\pi$ is the surjective map defined in (\ref{eqn: pi map}).
 \begin{definition}\label{defn:cluster G pi}
	We say two vectors $J_r,J_s$ with $1\leq r,s \leq k$ are connected if the integers $r,s$ are connected in $G_{\pi}$ via  a path. We say vectors $\{J_{r_1},\ldots ,J_{r_k}\}$ form a cluster if $\{r_1,\ldots , r_k\}$ is the vertex set of a cluster in $G_\pi$.
\end{definition}
Note that the definition of connectedness and cluster only depend on the graph $G_{\pi}$ and is independent of the choice of $j_1,j_2,\ldots, j_{pk/2}$.

For a given $\pi \in \mathcal{P}_2(pk)$ and an associated graph $G_{\pi}$, we define
\begin{equation}\label{eqn:m pi}
m(\pi) =  \# V_{\pi}-c= k-c,
\end{equation}
where $c$ is the number of connected components of $G_{\pi}$.
\section{Limiting moment sequence} \label{sec:wp_ind} 

We first outline the proof of Theorem \ref{thm:moments} and discuss its departure from the technique of the proof of Theorem 6.4 of \cite{liu2012fluctuations}, and then we prove Theorem \ref{thm:moments}. In Section \ref{subsec:exi+uniq_Gammap}, we discuss the existence and uniqueness of measure corresponding to the limiting moment sequence, and in Section \ref{subsec:polytest_ind}, we study the fluctuations of linear eigenvalue statistics of Hankel matrices for polynomial test functions.

 Now we explain how the case of odd degree monomial test functions is significantly difficult than the case of even degree monomials dealt in \cite{liu2012fluctuations}. Recall $w_p$ from (\ref{eqn:wp_Hn_odd}). Using Result \ref{res:trace_Hn}, we get 
\begin{equation}\label{eqn:Ewp1}
	\E[w_p^k]= \frac{1}{n^{\frac{pk}{2}}}\displaystyle \sum_{J_1, \ldots, J_k} I_{J_1}I_{J_2}\cdots I_{J_k}\Delta_{J_1}\Delta_{J_2}\cdots\Delta_{J_k}\E[x_{J_1}x_{J_2}\cdots x_{J_k}],
\end{equation}
where for each $r=1,2, \ldots, k$,
\begin{equation*} 
	J_r=(j_1^r, j_2^r,\ldots , j_p^r), \ I_{J_r} = \prod_{\ell=1}^p \chi_{[1, n]}\left(i_r-\sum_{q=1}^{\ell}(-1)^q j^r_{q}\right), \ x_{J_r}= \prod_{\ell=1}^{p} x_{j^r_\ell},
\end{equation*}
with $j_q^r \in \{-n,-n+1,\ldots, 0, \ldots , n\}$ and 
\begin{align*}
	\Delta_{J_r}=	
	\begin{cases}
		\displaystyle \delta_{2 i_r-n-1}\left( \sum_{q=1}^{p}(-1)^{q} j_{q}^r\right) & \quad \mbox{if } p \text{ is odd},\\
		\displaystyle \delta_{0}\left( \sum_{q=1}^{p}(-1)^{q} j_{q}^r\right) & \quad \text{if } p \text { is even}.
	\end{cases}
\end{align*}

For each choice of $J=(J_1,J_2,\ldots , J_k)$, we define $j_{(r-1)p+s} :=j^r_s$ for  $r=1,2,\ldots, k$ and $1 \leq s \leq p$.  Now we define a relation $\sim_J$ on $[pk]$ such that $r_1 \sim_J r_2$ if $j_{r_1}=j_{r_2}$. Then (\ref{eqn:Ewp1}) can be written as
\begin{equation}\label{eqn:Ewp2}
	\E[w_p^k]=\frac{1}{n^{\frac{pk}{2}}}\displaystyle \sum_{\pi \in \mathcal{P}(pk)} \sum_{J_\pi}I_{J_1}I_{J_2}\cdots I_{J_k}\Delta_{J_1}\Delta_{J_2}\cdots\Delta_{J_k}\E[x_{J_1}x_{J_2}\cdots x_{J_k}],
\end{equation}
where $\mathcal{P}(pk)$ is the set of all partitions of $[pk]$ and $J_\pi$ is the set of all possible $J=(J_1,J_2,\ldots, J_k)$ such that the block structure of $(J_1,J_2,\ldots, J_k)$ under the relation $\sim_J$ is $\pi$.
In the proof of Theorem \ref{thm:moments}, we show that only pair-partitions of $[pk]$ contribute to $\E[w_p^k]$ in the limit and thus we argue further with fixed pair-partitions $\pi$.

For a fixed pair-partition $\pi$, $\{ j_q^r\}$ has $\frac{pk}{2}$ free choices and we represent each choice as a point in $\R^{pk/2}$ given by $W=(j_1,j_2,\ldots, j_{\frac{pk}{2}})$. Similarly  $\{i_r\}$ in (\ref{eqn:Ewp1}) has $k$ choices which we represent by the point $V=(i_1,i_2,\ldots , i_k)$. Consider the $(k+pk/2)$ dimensional space,  represented in Figure \ref{fig:odd-even} as a plane, with a general element as (V,W).

For a pair-partition $\pi$, consider the following two systems of equations:
\begin{align}
	\sum_{q=1}^{p}(-1)^{q} j_{\pi((r-1)p+q)}&=0, \  \forall \ 1 \leq r \leq k, \label{eqn:sys p even}\\
	\sum_{q=1}^{p}(-1)^{q} j_{\pi((r-1)p+q)}&=2 i_r-1-n, \ \forall \ 1 \leq r \leq k \label{eqn:sys p odd},
\end{align}
where $\pi((r-1)p+q)$ is the image of $(r-1)p+q$ under the surjective map $\pi$ given in (\ref{eqn: pi map}).

Note that for the summand in (\ref{eqn:Ewp2}) to be non-zero for some point $(V,W) \in \R^{k+pk/2}$, the components of it should obey the systems of equations (\ref{eqn:sys p even}) and (\ref{eqn:sys p odd}) for even and odd cases respectively.
For odd $p$ and even $k$, the solution space of the systems of equations  (\ref{eqn:sys p odd}) is a $pk/2$-dimensional spaces represented in Figure 2 as line segment MN.

Two differences come up between the $p$ odd and $p$ even cases. The first difference is that for even case, the solution space of (\ref{eqn:sys p even}) is independent of $n$, whereas for the odd case, the solution space MN is changing with respect to $n$. The second and less obvious difference is the following: Suppose for even values of $p$ and $k$, $(j_1,j_2,\ldots, j_{pk/2})$ obeys the system of equations (\ref{eqn:sys p even}) for some choice of $V=(i_1,\ldots,i_k)$ and $n \in \N$. Then $(j_1,j_2,\ldots , j_{pk/2})$ is a solution of the system of equations (\ref{eqn:sys p even}) for every $m \geq n$.
This is not true for the odd $p$ case and in fact, a $(j_1, j_2, \ldots , j_{pk/2})$ cannot be a solution for both odd and even $n$ simultaneously. 
%

The challenge of our work is to control, these additional dependencies and show that $\E[w_p^k]$ still converges to a definite integral, just like in the even case. The first issue is taken care of by Proposition \ref{pro: R_n limit} and the parity dependence is taken care of by Lemma \ref{lemma:cyclomatic number}.
\begin{figure}[h]
	\centering
	\includegraphics[height=55mm, width =55mm ]{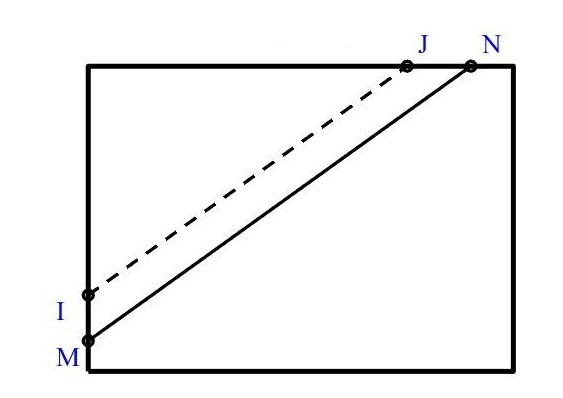}
	\caption{A pictorial representation of the $(k+pk/2)$-dimensional space corresponding to the solution space of the system of equations (\ref{eqn:sys p odd}). The line segment IJ corresponds to the solution space of $2i_r-1=\sum_{q=1}^{p} (-1)^{q} y_{\pi((r-1)p+q)}$ and the line segment MN corresponds to the solution space of $2i_r-1-\frac{1}{n}=\sum_{q=1}^{p} (-1)^{q} y_{\pi((r-1)p+q)}$.} \label{fig:odd-even}
\end{figure}

Notice in Figure \ref{fig:odd-even} that the distance between MN and IJ under the sup norm metric is $1/2n$ and therefore, the distance between them goes to zero as $n \rightarrow \infty$. Recall the definite integral $f_k(\pi)$ in Definition \ref{defn: partition integrals} . Considering $f_k(\pi)$ as a function on $\R^{pk/2+k}$, with the first $k$ components representing $\{x_r\}$ and the rest $pk/2$ components representing $\{y_i\}$, we obtain that $f_k(\pi)$ is an integral over IJ. The next proposition shows that $f_k(\pi)$ can be approximated by appropriate summations on the space MN.
\begin{proposition}\label{pro: R_n limit}
Let $p$ be a fixed odd natural number and $k$ be an even natural number.  For $n \in \N$ and $\pi \in \mathcal{P}_2(pk)$, if we define
\begin{equation}
\mathcal{R}_n := \frac{1}{n^{\frac{pk}{2}}}\displaystyle \mathlarger{\sum}_{j_1,j_2,\ldots, j_{pk/2}=-n}^n \prod_{r=1}^{k}\prod_{\ell=1}^p \chi_{[1, n]}\left(i_r-\sum_{q=1}^{\ell}(-1)^q j_{\pi\left((r-1)p+q\right)}\right),
\end{equation}
then as $n \rightarrow \infty$, $\mathcal{R}_n$ converges to $f_k(\pi)$, where  $i_r =\frac{1}{2}\left( \sum_{q=1}^{p}(-1)^q j_{\pi\left((r-1)p+q\right)}+n+1 \right)$ and $f_k(\pi)$ is as in Definition \ref{defn: partition integrals}.
\end{proposition}
\begin{proof}
For $n \in \N$ and $\pi \in \mathcal{P}_2(pk)$, we define 
$$\mathcal{T}_n= \frac{1}{n^{\frac{pk}{2}}}\displaystyle \sum_{j_1,j_2,\ldots, j_{pk/2}=-n}^n \prod_{r=1}^k\prod_{\ell=1}^p \chi_{[0, 1]}\left(\frac{\tilde{i}_r}{n}-\sum_{q=1}^{\ell} (-1)^q\frac{j_{\pi\left((r-1)p+q\right)}}{n}\right),$$
where $\tilde{i}_r =\frac{1}{2}\left( \sum_{q=1}^{p}(-1)^q j_{\pi\left((r-1)p+q\right)}+ n \right)$. Note that $\mathcal{T}_n$ is a Riemann sum  of the integral $f_k(\pi)$.

Now
\begin{align*}
&|\mathcal{R}_n - \mathcal{T}_n| \\ 
& = \bigg|  \frac{1}{n^{\frac{pk}{2}}} \hskip-8pt \displaystyle \mathlarger{\sum}_{j_1,j_2,\ldots, \atop j_{pk/2}=-n}^n \hskip-8pt \left(\prod_{r=1}^{k}\prod_{\ell=1}^p \chi_{[\frac{1}{n}, 1]}\left(\frac{i_r}{n}-\sum_{q=1}^{\ell} (-1)^q\frac{j_{\pi\left((r-1)p+q\right)}}{n}\right)-\prod_{r=1}^{k}\prod_{\ell=1}^p \chi_{[0, 1]}\left(\frac{\tilde{i}_r}{n}-\sum_{q=1}^{\ell} (-1)^q\frac{j_{\pi\left((r-1)p+q\right)}}{n}\right)\right)\bigg|\\
& \leq\frac{1}{n^{\frac{pk}{2}}}\displaystyle \hskip-8pt \mathlarger{\mathlarger{\sum}}_{j_1,j_2, \ldots, \atop j_{pk/2}=-n}^n \hskip-8pt \mathlarger{\mathlarger{\sum}}_{r=1}^{k}\mathlarger{\mathlarger{\sum}}_{\ell=1}^p\left|\left(\chi_{[\frac{1}{n}, 1]}\left(\frac{i_r}{n}-\sum_{q=1}^{\ell} (-1)^q\frac{j_{\pi\left((r-1)p+q\right)}}{n}\right)-\chi_{[0, 1]}\left(\frac{\tilde{i}_r}{n}-\sum_{q=1}^{\ell} (-1)^q\frac{j_{\pi\left((r-1)p+q\right)}}{n}\right)\right)\right|.
\end{align*}
Since $\frac{i_r}{n}=\frac{\tilde{i}_r}{n}+\frac{1}{2n}$, we have
\begin{equation*}
\chi_{[\frac{1}{n}, 1]}\left(\frac{i_r}{n}-\sum_{q=1}^{\ell} (-1)^q\frac{j_{\pi\left((r-1)p+q\right)}}{n}\right) = \chi_{\left[ \frac{1}{2n}, 1-\frac{1}{2n}\right]}\left(\frac{\tilde{i}_r}{n}-\sum_{q=1}^{\ell} (-1)^q\frac{j_{\pi\left((r-1)p+q\right)}}{n}\right).
\end{equation*} 
Thus for fixed $r$ and $\ell$, 
\begin{align}\label{eqn: xhi_sum}
& \left|\chi_{[\frac{1}{n}, 1]}  \left(\frac{i_r}{n}-\sum_{q=1}^{\ell} (-1)^q\frac{j_{\pi\left((r-1)p+q\right)}}{n}\right)-\chi_{[0, 1]}\left(\frac{\tilde{i}_r}{n}-\sum_{q=1}^{\ell} (-1)^q\frac{j_{\pi\left((r-1)p+q\right)}}{n}\right) \right|\nonumber\\
& = \left(\chi_{\left[0, \frac{1}{2n} \right]}\left(\frac{\tilde{i}_r}{n}-\sum_{q=1}^{\ell} (-1)^q\frac{j_{\pi\left((r-1)p+q\right)}}{n}\right)+\chi_{[1-\frac{1}{2n}, 1]}\left(\frac{\tilde{i}_r}{n}-\sum_{q=1}^{\ell} (-1)^q\frac{j_{\pi\left((r-1)p+q\right)}}{n}\right)\right).
\end{align}
We show that for each choice of $r,\ell$, the number of possibilities of $\{j_1,j_2,\ldots, j_{pk/2}\}$ such that the right side of (\ref{eqn: xhi_sum}) is non-zero is of the order $O(n^{pk/2-1})$. Observe that
\begin{align}\label{eqn:simplified_ir}
	\frac{\tilde{i}_r}{n}-\sum_{q=1}^{\ell} (-1)^q\frac{j_{\pi\left((r-1)p+q\right)}}{n} &= \frac{1}{2}\left(\sum_{q=1}^{p} (-1)^q\frac{j_{\pi\left((r-1)p+q\right)}}{n}+1\right)-\sum_{q=1}^{\ell} (-1)^q\frac{j_{\pi\left((r-1)p+q\right)}}{n}\nonumber\\
	&= \frac{1}{2}\left(1+\sum_{q=\ell+1}^{p} (-1)^q\frac{j_{\pi\left((r-1)p+q\right)}}{n}-\sum_{q=1}^{\ell} (-1)^q\frac{j_{\pi\left((r-1)p+q\right)}}{n}\right).
\end{align}
 
 Suppose $j_2,j_3,\ldots, j_{pk/2}$ are chosen. Since the length of the interval $[0,1/2n]$ is $1/2n$, we get that there exist at most two values of $j_1=j_{\pi(1)}$ such that the expression of (\ref{eqn:simplified_ir}) belongs to the interval $[0,1/2n]$. The same reasoning also implies that there exist at most only two values of $j_1$ such that the expression of (\ref{eqn:simplified_ir}) belongs to the interval $[1-\frac{1}{2n},1]$.
Therefore there is a reduction in the degree of freedom for choosing $j_k$'s. As a consequence,
\begin{equation*}
\frac{1}{n^{pk/2}}\hskip-3pt\sum_{j_1,j_2, \ldots, \atop j_{pk/2}=-n}^n\hskip-5pt \left|\chi_{[\frac{1}{n}, 1]} \left(\frac{i_r}{n}-\sum_{q=1}^{\ell} (-1)^q\frac{j_{\pi\left((r-1)p+q\right)}}{n}\right)-\chi_{[0, 1]}\left(\frac{\tilde{i}_r}{n}-\sum_{q=1}^{\ell} (-1)^q\frac{j_{\pi\left((r-1)p+q\right)}}{n}\right)\right|= O\left(\frac{1}{n}\right).
\end{equation*}
Thus, $|\mathcal{R}_n - \mathcal{T}_n|$ converges to zero. Since $\mathcal{T}_n$ is the Riemann sum of the integral $f_k(\pi)$, this proves our result.
\end{proof}
Now we prove Theorem \ref{thm:moments}.
\begin{proof} [Proof of Theorem \ref{thm:moments}]
First note from (\ref{eqn:wp_Hn_odd}) that $(w_p)^k= [\Tr(A_n^p)]^k.$
  Now using Result \ref{res:trace_Hn}, we get
\begin{equation}\label{eqn:E(w_p^k)}
\E[w_p^k]= \E\left[ \frac{1}{n^{p/2}}\sum_{J \in A_p} x_J I_J \right]^k =  \frac{1}{n^{\frac{pk}{2}}}\displaystyle \sum_{J_1\in A_p, \ldots, J_k \in A_p} I_{J_1}I_{J_2}\cdots I_{J_k}\E[x_{J_1}x_{J_2}\cdots x_{J_k}],
\end{equation}
 where $A_p$ is as in (\ref{eqn:Ap_HnOdd}) and for each $r=1,2, \ldots, k$,
 \begin{equation} \label{eqn:J_I_X}
J_r=(j_1^r, j_2^r,\ldots , j_p^r), \ I_{J_r} = \prod_{\ell=1}^p \chi_{[1, n]}\left(i_r-\sum_{q=1}^{\ell}(-1)^q j^r_{q}\right), \ x_{J_r}= \prod_{\ell=1}^{p} x_{j^r_\ell}.
 \end{equation}

 We claim that only pair-partitions of $[pk]$ contribute to (\ref{eqn:E(w_p^k)}). Consider $S_{J_r}=\{ j^r_1, j^r_2,\ldots , j^r_p \}$ for $1 \leq r \leq k$. For $(J_1,J_2,\ldots, J_k)$ to have non-zero contribution in (\ref{eqn:E(w_p^k)}), it is necessary that each element of $S_{J_1}\cup S_{J_2}\cup \cdots \cup S_{J_k}$ has multiplicity at least two, as $\E[x_j]=0$ for each $j$.

It follows from Corollary \ref{cor:cardinality Uk} that if any element has multiplicity greater than or equal to three, then the number of choices for choosing $(J_1,J_2,\ldots ,J_k)$ would be $O(n^{\lfloor\frac{pk-1}{2}\rfloor})$. Thus the contribution of all such terms to (\ref{eqn:E(w_p^k)}) would be $O\left(\frac{1}{n^{\frac{pk}{2}}} \times n^{\lfloor\frac{pk-1}{2}\rfloor}\right)= O(n^{-\frac{1}{2}})$. This shows that when $k$ is odd,  $\E[w_p^k]= o(1)$.

Now suppose $k$ is even. So, from the above discussion it is clear that for a contribution of the order $O(1)$ in $\E[w_p^k]$, we need to consider only pair-partitions. Since we have $\E[x_i^2]=1$ for all $i$, (\ref{eqn:E(w_p^k)}) can be rewritten as 
\begin{align}\label{eqn:E(wp^k pair-partition)}
\E[w_p^k] &= \displaystyle  \frac{1}{n^{pk/2}} \sum_{J_1,J_2,\ldots , J_{k}} \sum_{i_1,\ldots ,i_k=1}^n I_{J_1}I_{J_2}\cdots I_{J_k} \Delta_{J_1} \Delta_{J_2} \cdots \Delta_{J_k}\nonumber\\
&=\displaystyle  \frac{1}{n^{pk/2}}\sum_{\pi \in \PP_2(kp)} \sum_{j_1,j_2,\ldots , j_{pk/2}=-n \atop j_1\neq j_2\neq \cdots \neq j_{pk/2}}^n \sum_{i_1,\ldots ,i_k=1}^n I_{J_1}I_{J_2}\cdots I_{J_k} \Delta_{J_1} \Delta_{J_2} \cdots \Delta_{J_k},
\end{align}
where $I_{J_r}$ is as in (\ref{eqn:J_I_X}) and $\Delta_{J_r}=\delta_{2 i_r-n-1}\left( \sum_{q=1}^{p}(-1)^{q} j_{q}^r\right)$ with $j_q^r= j_{\pi \left((r-1)p+q\right)}$. Henceforth in this proof, we shall use $j_q^r$ to denote $j_{\pi \left((r-1)p+q\right)}$.
For a fixed $\pi \in \mathcal{P}_2(pk)$, consider the summation in (\ref{eqn:E(wp^k pair-partition)}) corresponding to $\pi$, that is,
\begin{equation}\label{eqn:E(w_p^k) of pi}
\sum_{j_1,j_2,\ldots , j_{pk/2}=-n \atop j_1\neq j_2\neq \cdots \neq j_{pk/2}}^n \sum_{i_1,\ldots ,i_k=1}^n I_{J_1}I_{J_2}\cdots I_{J_k} \Delta_{J_1} \Delta_{J_2} \cdots \Delta_{J_k}.
\end{equation}

Note that for a fixed $\pi$ and $j_1,j_2,\ldots, j_{pk/2}$, the term $\Delta_{J_r}$ is non-zero for some choice of $i_r$, if and only if both the following conditions are satisfied: 
\begin{enumerate}[(i)]
\item $-(n-1) \leq \sum_{q=1}^{p}(-1)^{q} j_q^r \leq n-1$,
\item $ \sum_{q=1}^{p}(-1)^{q}  j_q^r$ and $n$ have different parity.
\end{enumerate}
Therefore (\ref{eqn:E(wp^k pair-partition)}) can be written as 
\begin{equation*}
	\E[w_p^k]= \frac{1}{n^{pk/2}}\displaystyle \sum_{\pi \in \PP_2(kp)}\sum_{j_1,\ldots , j_{pk/2}=-n \atop j_1\neq \cdots \neq j_{pk/2}}^n \prod_{r=1}^k \left(I_{J_r} \ \chi_{[-(n-1),n-1]}\big\{\sum_{q=1}^p (-1)^q j_q^r\big\} \ \delta_{-1}\left( (-1)^{\sum_{q=1}^{p}(-1)^q j_q^r+n}\right) \right),
\end{equation*}
where $I_{J_r}$ is as in (\ref{eqn:J_I_X}) with $i_r$ being a function of $j_k$'s defined by $i_r=\frac{1}{2}\left( \sum_{q=1}^{p}(-1)^q j_q^r+n+1 \right)$.
Observe that $ \chi_{\left[1, n\right]}\left(i_r-\sum_{q=1}^{p}(-1)^q j_q^r\right)$ is non-zero for some $i_r$ only if (i) is satisfied. Thus, the above equation becomes
\begin{align}\label{eqn: Ewp^k xhi}
	\E[w_p^k]&=\frac{1}{n^{pk/2}}\displaystyle \sum_{\pi \in \PP_2(kp)}\sum_{j_1,\ldots , j_{pk/2}=-n \atop j_1\neq \cdots \neq j_{pk/2}}^n \prod_{r=1}^k \left(I_{J_r}\ \delta_{-1}\left( (-1)^{\sum_{q=1}^{p}(-1)^q j_q^r+n}\right)\right)\nonumber\\
&=\frac{1}{n^{pk/2}}\displaystyle \sum_{\pi \in \PP_2(kp)}\sum_{j_1,\ldots , j_{pk/2}=-n}^n \prod_{r=1}^k \left(I_{J_r}\ \delta_{-1}\left( (-1)^{\sum_{q=1}^{p}(-1)^q j_q^r+n}\right)\right)+o(1),
\end{align}
 where the last equality follows from Corollary \ref{cor:cardinality Uk} and our earlier observation that the contribution of partitions other than pair-partitions is $o(1)$.

For the rest of the proof, we fix a partition $\pi \in \mathcal{P}_2(pk)$. For the chosen $\pi \in \mathcal{P}_2(pk)$, consider $(J_1, J_2, \ldots , J_k)$ such that $J_{r}=\left(j_{1}^{r}, j_{2}^{r}, \ldots, j_{p}^{r}\right)$ where $j_q^r=j_{\pi \left((r-1)p+q\right)}$. For $s \neq r$, we define $J_{r,s}$ as the set of all cross-matched elements between $J_r$ and $J_s$, and $J_{r,r}$ as the set of all self-matched elements in $J_r$. As $\sum_{J_{r,r}} (-1)^qj_q^r$ is even, the parity of $\sum_{q=1}^p (-1)^q j_q^r$ is determined by
$\{\sum_{s \neq r} \sum_{J_{r,s}} (-1)^qj_q^r \}.$ Condition (ii) imposes that for all $r$, $\sum_{q=1}^p j_q^r$ must have same parity. We fix a particular combination of parity for $\sum_{j_q^r \in J_{r,s}} (-1)^q j_q^r$ such that $\sum_{s \neq r} \sum_{J_{r,s}}(-1)^q j_q^r$ is odd (or even) for all $r$. This problem can be translated to assigning a signed labelling for $G_\pi$. Notice that choosing the parity of $\sum_{J_{r,s}} (-1)^q j_q^r$ as odd (or even) is equivalent to labelling the edge $(J_r, J_s) \in E_{\pi}$ as $-1$ (or +1) (see Definition \ref{def:odd_even_labelling}). Hence, the condition $\sum_{s \neq r} \sum_{J_{r,s}}(-1)^q j_q^r$ is odd (or even) for all $r$, is equivalent to finding an all-odd (or all-even) labelling on $G_\pi$. Therefore, (\ref{eqn: Ewp^k xhi}) becomes
\begin{align}\label{eqn:Ew_p^k labelling}
	\E[w_p^k]= 
	\begin{cases}
		&\frac{1}{n^{pk/2}}\displaystyle \sum_{\pi \in \PP_2(kp)} \sum_{\gamma \in  G_\pi^{(ao)}}\sum_{j_1,\ldots , j_{pk/2} }\prod_{r=1}^k I_{J_r}+o(1)\quad \text{if }  n \text{ is even},\\
		&\frac{1}{n^{pk/2}}\displaystyle \sum_{\pi \in \PP_2(kp)}\sum_{\gamma \in  G_\pi^{(ae)}}\sum_{j_1,\ldots , j_{pk/2}} \prod_{r=1}^k I_{J_r} 
		+o(1) \quad \text{if } n \text{ is odd},
	\end{cases}
\end{align}
where  $G_\pi^{(ao)}$ and $G_\pi^{(ae)}$ are the set of all-odd labellings and all-even labellings of graph $G_\pi$, respectively.
In the above expression, the summation is taken over all $j_1,j_2,\ldots, j_{pk/2}$ obeying the system of equations
\begin{equation}\label{eqn:sys_labelling}
	(-1)^{\sum_{J_{r,s}} (-1)^q j_q^r}= \gamma\left((r,s)\right)\quad \text{ for all } (r,s) \in E_{\pi}.
\end{equation}

Let $\gamma$ be either an all-odd labelling or an all-even labelling of $G_\pi$.
Consider an edge $(r,s) \in E_\pi$. Let $j_u^{r,s} \in \{j_1,j_2, \ldots, j_{pk/2}\}$ be a cross-matched element in $(J_r,J_s)$. Observe that the equations in (\ref{eqn:sys_labelling}) are independent in the sense that each $j_q$ appears in at most one equation in (\ref{eqn:sys_labelling}). Therefore, once all other $j_q \in J_{r,s}$ except $j_u^{r,s}$ are fixed, the parity of $j_u^{r,s}$ is either odd or even, depending on the label $\gamma(r,s)$.
Furthermore the possible values of $j_u^{r,s}$ such that $\prod_{r}I_{J_r}$ is non-zero, belongs to an interval. As a consequence for each $(r,s) \in E_\pi$
$$\# \{ (J_1, J_2, \ldots, J_k): \prod_{r}I_{J_r}=1, j_u^{r,s} \text{ is odd}\}=\# \{ (J_1, J_2, \ldots, J_k):\prod_{r}I_{J_r}=1, j_u^{r,s} \text{ is even}\}+o(1).$$
%

Thus for a fixed all-odd (all-even) labelling, the reduction on number of possibilities of $(J_1,J_2, \ldots , J_k)$ is by a factor of $1/2^{\# E_\pi}$, with an error term of the order $o(1)$. Hence the contribution due to each edge-labelling is 
$$\sum_{j_1,j_2,\ldots , j_{pk/2}=-n}^n \frac{1}{2^{\# E_{\pi}}}\prod_{r=1}^k I_{J_r} + o(1),$$ 
which does not depend on $\gamma$.
Note from Lemma \ref{lemma:cyclomatic number} that the number of all-odd (and all-even) labellings of a cluster of $\pi$ is $2^{\#E-\#V+1}$, where $E$ and $V$ are the edge set and the vertex set of the cluster, respectively. Therefore the number of all-odd (all-even) edge labellings of $G_{\pi}$ is $2^{\#E_\pi-\#V_\pi+c}$, where $c$ is the number of connected components of $G_\pi$. Hence we get
\begin{align*}
\E[w_p^k]&= \frac{1}{n^{pk/2}}\displaystyle  \sum_{\pi \in \PP_2(kp)}2^{\#E_\pi-\#V_\pi+c}\sum_{j_1,j_2,\ldots , j_{pk/2}=-n}^n \frac{1}{2^{\#E_\pi}}\prod_{r=1}^k I_{J_r} +o(1)\\
&=\sum_{\pi \in \PP_2(kp)}\frac{1}{2^{m(\pi)}}\times \frac{1}{n^{pk/2}}\sum_{j_1,j_2,\ldots , j_{pk/2}=-n}^n \prod_{r=1}^k I_{J_r}+o(1).
\end{align*} 
Finally, (\ref{eqn:lim_moments_wp}) follows from Proposition \ref{pro: R_n limit}. This completes the proof.
\end{proof}
In the following section, we discuss the existence and uniqueness of the measure corresponding to the moment sequence $\{\beta_{k}\}$,  where $\{\beta_k\}$ is as in (\ref{eqn:lim_moments_wp}).
\subsection{Existence and uniqueness of measure:} \label{subsec:exi+uniq_Gammap}
It is known from Hamburger's theorem (Theorem 3.8, \cite{Konrad}) that a sequence $\{m_k\}$ is a moment sequence of a measure on $\mathbb{R}$ if and only if $\{m_k\}$ is a positive semi-definite sequence. 
We observe that  $\{\beta_k\}$ of Theorem \ref{thm:moments} is a positive semi-definite sequence. Consider a finite sequence $\{ c_1, c_2,\ldots , c_k \}$ of complex numbers. Then,
\begin{equation*}
	\sum_{i, j=1}^{k} c_{i} \bar{c}_{j} \beta_{i+j}=\sum_{i, j=1}^{k} c_{i} \bar{c}_{j}\lim_{n \rightarrow \infty} \E w_p^{i+j}=\lim_{n \rightarrow \infty} \sum_{i, j=1}^{k} c_{i} \bar{c}_{j} \E w_p^{i+j}.
\end{equation*}
Since $\{\E w_p^{i}\}$ is a moment sequence, $\sum_{i, j=1}^{k} c_{i} \bar{c}_{j} \E w_p^{i+j} \geq 0$ for each $n$ and as a result $\{\beta_{k}\}$ is a positive semi-definite sequence. Hence there exist a measure corresponding to $\{\beta_{k}\}$.
In the following lemma, we discuss the uniqueness of the moment sequence $\{\beta_{k}\}$.
\begin{lemma} \label{lem:existance_gammap}
	 Let $\{\beta_k\}$ be the moment sequence as given in (\ref{eqn:lim_moments_wp}). Then $\{\beta_{k}\}$ fails to obey Carleman's condition (Lemma 1.2.2, \cite{bose_patterned}).
\end{lemma}
 \begin{proof}
 Consider $f_k(\pi)$ and $U_r(p)$ as defined in Definition \ref{defn: partition integrals}. Notice that for an even $k$, pair-partition $\pi \in \mathcal{P}_2(pk)$ and $y_1, y_2, \ldots, y_{\frac{pk}{2}} \in \left[-\frac{1}{8p}, \frac{1}{8p}\right]$,  $x_r$ defined as in Definition \ref{defn: partition integrals} belong to the range $\left[\frac{3}{8},\frac{5}{8}\right]$ for all $1 \leq r \leq k$. This implies that for $y_1, y_2, \ldots, y_{\frac{pk}{2}} \in \left[-\frac{1}{8p}, \frac{1}{8p}\right]$, $U_r(p)=1$ for each $r$ and consequently, $f_k(\pi) \geq \left(\frac{1}{4p}\right)^{\frac{pk}{2}}$ and hence
 \begin{align} \label{eqn:beta_gamma+moment}
 	\beta_{2k} &=  \sum_{ \pi \in \mathcal{P}_2(2pk)} \frac{1}{2^{m(\pi)}}f_{2k}(\pi) \nonumber \\
 	& \geq \left(\frac{1}{4p}\right)^{pk} \times \frac{1}{2^{2k}} \# \mathcal{P}_2(2pk) = \left(\frac{1}{4p}\right)^{pk} \times \frac{1}{2^{2k}} \frac{(2pk)!}{2^{pk}(pk)!}= \gamma_{2k}, \mbox{ say}.
 \end{align}
 It follows from Stirling's approximation that the sequence $\{\gamma_{2k}\}$ does not obey  Carleman's condition. As a consequence of the inequality (\ref{eqn:beta_gamma+moment}), $\{\beta_k\}$ also fails to obey  Carleman's condition. 	
 \end{proof}


In spite of this, if we additionally assume that $\Gamma_{p}$ is the unique distribution with moment sequence $\{\beta_{k}\}$, then by Theorem \ref{thm:moments} and moment method, $w_{p} \stackrel{d}{\rightarrow}  \Gamma_p$. Clearly, the limit $\Gamma_p$ is universal. Later, in Section \ref{sec:prop_Gammap}, we show that $\Gamma_p$ is non-Gaussian and has unbounded support.

In the following section, we study the fluctuations of linear eigenvalue statistics of Hankel matrices for polynomial test functions.
\subsection{Polynomial test function:} \label{subsec:polytest_ind}
First we recall from (\ref{eqn:wp_Hn_odd}) that for $Q(x)= \sum_{d=1}^{p} c_d x^d$,
\begin{align*}
	\Tr[Q(A_n)] = \sum_{d=1}^{p} c_d \Tr(A_n^d)= \sum_{d=1}^{p} c_dw_d= w_Q, \mbox{ say}.
\end{align*}

The following lemma provides the order of convergence of $\E[w_{p_1} w_{p_2}]$ when one of $p_i$ is even and other is odd.
\begin{lemma} \label{lem:Ew_evenodd}
	 Let $p_1$ be even and $p_2$ be odd positive integers, then $\E[w_{p_1} w_{p_2}] =O(\sqrt{n}).$
\end{lemma}

\begin{proof}
 From Result \ref{res:trace_Hn}, we get 
\begin{align} \label{eqn:E[wp1wp2]}
	\E[w_{p_1} w_{p_2}] =  \frac{1}{n^{\frac{p_1+p_2}{2}-1}} \sum_{J_1 \in A'_{p_1},J_2 \in A_{p_2}} \E[x_{J_1} x_{J_2}] I_{J_{1}} I_{J_{2}},
\end{align}
where $I_{J_r}, x_{J_r}$ are as in (\ref{eqn:J_I_X}), $A_{p_2}$ is as in (\ref{eqn:Ap_HnOdd}) and $A'_{p_1}$ is defined as 
$$  A'_{p_1}=\Big\{(j_{1}, j_2, \ldots, j_{p_1}) \in\left\{0, \pm 1, \ldots, \pm n\right\}^{p_1}: \sum_{q=1}^{p_1} (-1)^{q} j_{q}=0\Big\}.$$
Since the entries $x_i$ satisfy condition (\ref{eqn:condition}) and $p_2$ is odd, $\E[x_{J_1} x_{J_2}]$ has the maximum contribution if the entries of $J_1$ are pair-matched and one entry of $J_2$ is triple matched with the rest entries as pair-matched. Thus
\begin{equation*}
	\sum_{J_1 \in A'_{p_1},J_2 \in A_{p_2}} |\E[x_{J_1} x_{J_2}] I_{J_{1}} I_{J_{2}}| = O(n^{\frac{p_1}{2} + \frac{p_2-1}{2} }).
\end{equation*}
Now on combining the above expression with (\ref{eqn:E[wp1wp2]}), we get $\E[w_{p_1} w_{p_2}] =O(\sqrt{n}).$	
\end{proof}

In general, if at least one of $p_i$ is even for a given set of positive integers $p_1, p_2, \ldots, p_k$, then
\begin{equation*}
	\E[w_{p_1} w_{p_2} \cdots w_{p_k}] \geq O(\sqrt{n}).
\end{equation*}
This shows that if $Q(x)$ is a polynomial with at least one even degree term, then for every $k\geq 1$, $\E[\Tr[Q(A_n)]]^k$ is divergent.

Now suppose $\displaystyle Q(x)= \sum_{d=1, \atop d \ odd}^{p} c_d x^d$ is a polynomial with odd degree terms only. Then
\begin{align} \label{eqn:E[wQ]^k}
	\E[w_Q]^k = \sum_{D_k} c_{p_1} \cdots c_{p_k} \E[w_{p_1} w_{p_2} \cdots w_{p_k}],
\end{align}
where $D_k = \{(p_1, p_2, \ldots, p_k) : p_i \in \{1,3, \ldots, p\} \mbox{ for } \ i=1,2, \ldots,k \}.$

By the similar arguments as used in Theorem \ref{thm:moments}, we get 
\begin{equation} \label{eqn:lim_moments_wp_poly}
	\Gamma_{p_1,p_2, \ldots, p_k} := \lim_{n \rightarrow \infty} \E[w_{p_1} w_{p_2} \cdots w_{p_k}]=
	\begin{cases}
		\displaystyle \sum_{ \pi \in \mathcal{P}_2(p_1+ \cdots +p_k)} \frac{1}{2^{m(\pi)}}h_k(\pi) & \quad \mbox{if } k \text{ is even},\\
		0 & \quad \text{if } k \text { is odd},
	\end{cases}
\end{equation}
where $ h_k(\pi) $ will be some definite integral similar to  $f_k(\pi)$, given in Definition \ref{defn: partition integrals} and $m(\pi)$ is as given in (\ref{eqn:m pi}).
Using the above equation in (\ref{eqn:E[wQ]^k}), we get 
\begin{equation} \label{eqn:lim_moments_wQ}
	\tilde{\beta}_k	:= \lim_{n \rightarrow \infty} \E[w_{Q}]^k=
	\begin{cases}
		\displaystyle \sum_{D_k} c_{p_1} \cdots c_{p_k} \Gamma_{p_1,p_2, \ldots, p_k} & \quad \mbox{if } k \text{ is even},\\
		0 & \quad \text{if } k \text { is odd}.
	\end{cases}
\end{equation}

Note from (\ref{eqn:lim_moments_wQ}) that $\{\tilde{\beta}_k\}$ is a positive semi-definite sequence and therefore there exist measures on $\mathbb{R}$ with $\{\tilde{\beta}_k\}$ as its moment sequence. If there exists a unique distribution say, $\Gamma_Q$, which corresponds to $\{\tilde{\beta}_k\}$, then from the moment method, we have 
$$ w_{Q} \stackrel{d}{\rightarrow}  \Gamma_Q  \ \ \ \mbox{ if $Q(x)$ has only odd degree terms},$$
where $\E[\Gamma_Q]^k= \tilde{\beta}_k$. 

\section{Properties of $\Gamma_p$} \label{sec:prop_Gammap}
In this section, we study some properties of $\Gamma_p$, where $\Gamma_p$ is any distribution on $\mathbb{R}$ with moment sequence $\{\beta_k\}$. The following proposition shows that the moments of $\Gamma_p$ dominate the moments of the Gaussian distribution. 
\begin{proposition} \label{pro:not_normal_Hn}
A distribution $\Gamma_p$ with moment sequence given by (\ref{eqn:lim_moments_wp}) is a non-Gaussian distribution.
\end{proposition}

\begin{proof}
First we recall that the moment sequence of $\Gamma_p$ is $\{\beta_k\}$, where $\{\beta_k\}$ is given in (\ref{eqn:lim_moments_wp}). Note that to show $\Gamma_p$ is non-Gaussian, it suffices to show that 
$\beta_{4} \neq 3 \beta_2^2$.

Now we calculate $\beta_4$. From (\ref{eqn:E(w_p^k)}), we have
\begin{equation} \label{eqn:E(w_p^4)}
\E[w_p^{4}]= \frac{1}{n^{2p}}\displaystyle \sum_{ A_p, A_p, A_p, A_p} \E[x_{J_1}x_{J_2} x_{J_3} x_{J_{4}}]  I_{J_1}I_{J_2}I_{J_3} I_{J_{4}},
\end{equation}
where $J_r \in A_p$ with $A_{p}$ as in (\ref{eqn:Ap_HnOdd}) and $J_r, I_{J_r}$, $x_{J_r}$ are as given in (\ref{eqn:J_I_X}) for each $r=1,2, 3,4$.


Recall the notions of connectedness and cluster from  Definition \ref{defn:cluster G pi}. Depending on connectedness between $J_{r}$'s, the following three cases arise in (\ref{eqn:E(w_p^4)}).
\vskip5pt
\noindent \textbf{Case I.} \textbf{At least one of $J_{r}$ for $r=1,2,3,4$, is not connected with the remaining ones:} 
Without loss of generality, suppose $J_{1}$ is not connected with other $J_{r}$'s. Then from the independence of  ${x_i}$, we get 
\begin{equation*}
\frac{1}{n^{2p}}\displaystyle \sum_{ A_p, A_p, A_p, A_p} \E[x_{J_1}x_{J_2} x_{J_3} x_{J_{4}}]  I_{J_1}I_{J_2}I_{J_3} I_{J_{4}} = \frac{1}{n^{2p}} \sum_{ A_p} \E[x_{J_1}]  I_{J_1} \sum_{ A_p, A_p, A_p}  \E[x_{J_2} x_{J_3} x_{J_{4}}]I_{J_2}I_{J_3} I_{J_{4}}.
\end{equation*}
Note that for each $i$, $\E[x_i]=0$ and that $J_1$ has odd many components.  Corollary \ref{cor:cardinality Uk} implies that $\sum_{ A_p} \E[x_{J_1}] I_{J_1}$ can have maximum contribution of the order $O(n^{\frac{p-1}{2}})$. Again, using Corollary \ref{cor:cardinality Uk}, we can show that $\sum_{ A_p, A_p, A_p}  \E[x_{J_2} x_{J_3} x_{J_{4}}] I_{J_2}I_{J_3} I_{J_{4}}$ can has contribution of the order at most $O(n^{\frac{3p-1}{2}})$. 
Thus, if $T_1$ is the contribution of this case to $\E[w_p^{4}]$, then
\begin{equation}\label{eqn:T1_nnormal}
T_1= o(1).
\end{equation}
 \vskip5pt
 \noindent \textbf{Case II.} \textbf{$J_{1}$ is connected with only one of $J_{2}, J_{3}, J_{4}$ and the remaining two of $J_{2}, J_{3}, J_{4}$  are connected only with each other:} Without loss of generality, we assume $J_{1}$ is connected only with $J_{2}$ and $J_{3}$ is connected only with $J_{4}$.  
So, from the independence of $\{x_i\}$, we get
\begin{align}\label{eqn:E_split_T2}
\frac{1}{n^{2p}}\displaystyle \sum_{ A_p, A_p, A_p, A_p} \E[x_{J_1}x_{J_2} x_{J_3} x_{J_{4}}]  I_{J_1}I_{J_2}I_{J_3} I_{J_{4}}
&= \frac{1}{n^{2p}} \sum_{A_p, A_p} \E[x_{J_1} x_{J_2}]  I_{J_1}I_{J_2} \sum_{A_p, A_p}  \E[ x_{J_3} x_{J_{4}}]  I_{J_3}I_{J_4} \nonumber \\
& = \E[w_p^{2}] \E[w_p^{2}].
\end{align}
Observe that in this case, there are two more subcases: 
\begin{itemize}
\item [(i)] $J_{1}$ is connected only with $J_{3}$ and $J_{2}$ is connected only with $J_{4}$,
\item[(ii)] $J_{1}$ is connected only with $J_{4}$ and $J_{2}$ is connected only with $J_{3}$.
\end{itemize}
 Therefore if we denote the contribution of this case to $\E[w_p^{4}]$ by $T_2$, then from (\ref{eqn:E_split_T2}), we get
\begin{equation}\label{eqn:T2_nnormal}
T_2= 3 \big(\E[w_p^{2}]\big)^2.
\end{equation}
\vskip5pt
 \noindent \textbf{Case III.} \textbf{$\{J_{1}, J_{2}, J_{3}, J_{4}\}$ form a cluster:}  
 Suppose $T_3$ is the contribution of this case to $\E[w_p^{4}]$. Then from the independence of $\{x_i\}$ and $\E(x_{i})=0$, we get
\begin{align} \label{eqn:T3} 
T_3
 =  \frac{1}{n^{2p}} \sum_{( J_{1}, J_{2}, J_{3}, J_{4}) \in \tilde{B}_{P_4}}  \E[x_{J_1}x_{J_2} x_{J_3} x_{J_{4}}]I_{J_1}I_{J_2}I_{J_3} I_{J_{4}},
\end{align}
 where  $J_r=(j^r_1, j^r_2, \ldots, j^r_p)$ for each $r=1,2,3,4$ and  $\tilde{B}_{P_4}$ is defined as,
 \begin{align*} 
\tilde{ B}_{p_4}=\{(J_{1}, J_{2}, J_{3}, J_{4}) \in & \ A_{p} \times A_{p} \times A_{p} \times A_{p}  : \{J_{1}, J_{2}, J_{3}, J_{4}\}  \mbox{ form a cluster and entries of }  \\
& \qquad  S_{J_{1}} \cup S_{J_{2}} \cup S_{J_{3}} \cup S_{J_{4}}  \mbox{ have multiplicity greater than or equal to two}
\}.
\end{align*}
Now consider $B'_{P_4}$, a subset of $\tilde{B}_{P_4}$ such that $(J_{1}, J_{2}, J_{3}, J_{4}) \in B'_{P_4}$ if
\begin{itemize}
\item[(i)] $j_1^1=j^2_1, j^1_2= j_1^3, j^1_3= j_1^4$, $j_1^1 \neq j^1_2 \neq j^1_3$, 
\item[(ii)] $j^1_{2d}= j^1_{2d+1}$ for all $d=2,3, \ldots, \frac{p-1}{2}$ and for $r=2,3,4, \ j^r_{2d}= j^r_{2d+1}$ for all $d=1,2, \ldots, \frac{p-1}{2}$,
\item[(iii)] $I_{J_1}I_{J_2}I_{J_3} I_{J_{4}}=1$.
\end{itemize}
\begin{localclaim}\label{claim: card_Bp4}
$\# B'_{P_4} \geq  \frac{n^{2p}}{8^3}$.
\end{localclaim}
\textbf{Proof.}  Note that if $(J_{1}, J_{2}, J_{3}, J_{4}) \in B'_{P_4}$, then from condition (ii), the entries of $J_r$ has the following constraints:
\begin{align} \label{eqn:j_i_constraint}
-j_1^1 +j_2^1 -j_3^1 = 2i_1-1-n, \ -j_1^1 = 2i_2-1-n, \ -j_2^1 = 2i_3-1-n, \  -j_3^1 = 2i_4-1-n,
\end{align}
where for a fixed $n$, $j_d^r \in \{0, \pm 1, \pm 2, \ldots, \pm n \}$ and $i_r \in \{1, 2, \ldots, n \}$ for each $r=1,2,3,4$. Also observe from the condition (iii) that
\begin{equation*}
	I_{J_r} = \prod_{\ell=1}^p \chi_{[1, n]}\left(i_r-\sum_{d=1}^{\ell} (-1)^d j^r_{d}\right)= 1 \ \ \forall \ r=1,2,3,4,
\end{equation*}
which shows that for all $r=1,2,3,4$, we have to choose $j_1^r, j^r_2, \ldots, j^r_p$ from $\{0, \pm 1, \pm 2, \ldots, \pm n \}$ and $i_r$ from $\{1, 2, \ldots, n \}$
such that 
\begin{equation*}
	i_r-\sum_{d=1}^{\ell} (-1)^d j^r_{d} \in \{1,2, \ldots,n\}  \ \ \forall \ r=1,2,3,4 \mbox{ and} \ \forall \ \ell=1,2, \ldots, p.
\end{equation*}

Now we calculate cardinality of $B'_{P_4} $. First note from (\ref{eqn:j_i_constraint}) that $i_2-i_3+i_4=i_1$. So, if we choose $i_1,i_2,i_3$ freely, then $i_4$ will be fixed.
Let 
\begin{align} \label{eqn:theta_i}
	\theta = \Big\{ (i_1,i_2,i_3,i_4) \in \big\{ [7n/16,9n/16] \cap \Z \big\}^3 \times \{1,2, \ldots, n\} \suchthat i_2-i_3+i_4=i_1  \Big\}.
\end{align}
Then $\# \theta \geq (\frac{n}{8})^3$ (by choosing  $i_1,i_2,i_3$ freely), where $[x,y] \cap \Z$ denotes the set of integers between $x$ and $y$. For simplicity of notation, we write $[x,y]$ in place of $[x,y]\cap \Z$.

 First we calculate the contribution from $J_2$. Recall from the condition (ii) that  $j^2_2=j^2_3, j^2_4=j^2_5, \ldots, j^2_{p-1}=j^2_p$. Note that, once $i_2$ is chosen, $j^2_1$ will be fixed by $-j_1^2 = 2i_2-1-n$ with $j^2_1 \in [-1- n/8, n/8-1] \subseteq [-n, n]$. Also $i_2+j^2_1  \in \{1,2, \ldots, n\}$.  Now we choose $j^2_2$ freely in $n$ ways from the range $[c^2_1-n, c^2_1-1] \subseteq [-n, n]$, where $c^2_1= i_2+j^2_1$. Then $i_2+j^2_1-j^2_2, \ i_2+j^2_1-j^2_2+j^2_3  \in \{1,2, \ldots, n\}$ ($j^2_2=j^2_3$). Once $j^2_2$ is chosen, we choose $j^2_4$ freely  in $n$ ways from the range $[c^2_1-n, c^2_1-1]$. Here  note that $i_2-\sum_{d=1}^{4} (-1)^d j^2_{d}, \  i_2-\sum_{d=1}^{5} (-1)^d j^2_{d}  \in \{1,2, \ldots, n\} $.  By continuing this idea, we can show that $i_2-\sum_{d=1}^{\ell} (-1)^d j^2_{d}  \in \{1,2, \ldots, n\} $ for each $\ell=6,7, \ldots, p$ and the total number of degree of freedom is $n^{\frac{p-1}{2}}$. Similarly, we can show that for each $J_3$ and $J_4$, we have  $i_3-\sum_{d=1}^{\ell} (-1)^d j^3_{d}, \ i_4-\sum_{d=1}^{\ell} (-1)^d j^4_{d}  \in \{1,2, \ldots, n\} $ for each $\ell=1, 2, \ldots, p$ and for each $J_3, J_4$, the total number of degree of freedom is $n^{\frac{p-1}{2}}$.
  
  Now we calculate the contribution from $J_1$. Recall from the condition (i) that  $j_1^1=j^2_1, j^1_2= j_1^3, j^1_3= j_1^4$. So, once $J_2,J_3,J_4$ are chosen, $j^1_1, j^1_2, j^1_3$ will be fixed with  $j^1_r \in [-1- n/8, n/8-1]$ for all $r=1,2,3$. Note that $i_1,i_2,i_3,i_4$ are also chosen as $(i_1,i_2,i_3,i_4) \in \theta$, where $\theta$ is given in (\ref{eqn:theta_i}). Therefore from the above range of $j_r^1$ and $i_r$, we can show that, $i_1+j^1_1, i_1+j^1_1-j^1_2, \ i_1+j^1_1-j^1_2+j^1_3  \in \{1,2, \ldots, n\}$. Now we choose $j^1_4, \ldots, j^1_p$.
 Note from the condition (ii) that  $j^1_4=j^1_5, j^1_6=j^1_7, \ldots, j^1_{p-1}=j^1_p$. Therefore from the similar idea as used to calculate carnality of $J_2$, we can show that $i_1-\sum_{d=1}^{\ell} (-1)^d j^1_{d}  \in \{1,2, \ldots, n\} $ for each $\ell=4,5, \ldots, p$ and the total number of degree of freedom in $J_1$ is $n^{\frac{p-3}{2}}$.
 Hence 
\begin{align}\label{eqn:card_B'P4}
\# B'_{P_4} =  \theta \big[ n^{\frac{p-3}{2}} \big] \big[ n^{\frac{p-1}{2}} \big]^3 \geq  \frac{n^{2p}}{8^3},
\end{align}
where the above inequality arises due to  $\theta \geq (\frac{n}{8})^3$.

On using the fact that the entries are independent with mean zero and variance one, from  (\ref{eqn:T3}) we get
\begin{align} \label{eqn:T3_nnormal}
\lim_{n\to\infty} T_3 \geq  \lim_{n\to\infty}  \frac{1}{n^{2p}} \# B'_{P_4} \geq \frac{1}{8^3} >0,
\end{align}
where the last inequality arises due to (\ref{eqn:card_B'P4}).

Now combining (\ref{eqn:T1_nnormal}), (\ref{eqn:T2_nnormal}) and  (\ref{eqn:T3_nnormal}) with (\ref{eqn:E(w_p^4)}), we get $\beta_4 > 3\beta_2^2.
$
This completes the proof of Proposition \ref{pro:not_normal_Hn}.
\end{proof}
The following corollary is an easy consequence of Proposition \ref{pro:not_normal_Hn} and Theorem 4.5.2 of \cite{KLChung}. 
\begin{corollary}\label{cor:non-Gaussian_limit}
For each odd $p \geq 3$, $w_p$ does not converge in distribution to a Gaussian random variable, where $w_p$ is as in (\ref{eqn:wp_Hn_odd}).
\end{corollary}
Our next proposition provides the unbounded support property of $\Gamma_p$.

\begin{proposition} \label{pro:unbdd_supp_Hn}
A distribution $\Gamma_p$ with moment sequence given by (\ref{eqn:lim_moments_wp}) has unbounded support.
\end{proposition}

\begin{proof}
First note that to show $\Gamma_p$ has unbounded support, it suffices to show that 
$(\beta_{2k})^{\frac{1}{k}} \tends  \infty$. Now we recall from (\ref{eqn:E(w_p^k)}) that
\begin{equation} \label{eqn:E(w_p^k)2}
\E[w_p^{2k}]= \frac{1}{n^{pk}}\displaystyle \sum_{J_1\in A_p, \ldots, J_{2k} \in A_p} \E[x_{J_1}x_{J_2}\cdots x_{J_{2k}}] I_{J_1}I_{J_2}\cdots I_{J_{2k}}.
\end{equation}
where  for $1 \leq r \leq 2k$, $J_r, I_{J_r}$ and $x_{J_r}$ are as in (\ref{eqn:J_I_X}) and $A_{p}$ is as in (\ref{eqn:Ap_HnOdd}).

Observe from the independence of the entries $\{x_i\}$ that if there exist an $\ell \in \{1,2, \ldots,2k\}$ such that $J_\ell$ is not connected with any other $J_r$, then
\begin{align*}
 \sum_{A_p, \ldots, A_p} |\E[x_{J_1}x_{J_2}\cdots x_{J_{2k}}] I_{J_1}I_{J_2}\cdots I_{J_{2k}}|	& \leq \sum_{ A_p} |\E[x_{J_1}]| \sum_{ A_p, \ldots, A_p}  |\E[x_{J_2} \cdots x_{J_{2k}}] | \\
 & \leq O(n^{\lfloor\frac{p}{2}\rfloor}) O(n^{\lfloor\frac{ 2pk-p}{2}\rfloor}),
\end{align*}
which shows that this case has contribution of the order $o(1)$ in $\E[w_p^{2k}]$.
Note that the last inequality in the above expression arises due to the uniform boundedness of moments of $\{x_i\}$ and Lemma \ref{lem:card B_p}.

The above observation shows that the cases which have non-zero contribution in $\lim_{n\to\infty}\E[w_p^{2k}]$ are the cases when $\{J_1, J_2, \ldots, J_{2k}\}$ decomposes into clusters of size at least two. Hence 
\begin{align*}
 \lim_{n\to\infty}\E[w_p^{2k}] & \geq  \lim_{n\to\infty}\frac{1}{n^{pk}} \sum_{\pi \in \mathcal P_2(2k)} \prod_{i=1}^{k}  \sum_{ J_{y(i)} \in A_{p},\ J_{z(i)} \in A_{p}} \E\big[x_{J_{y(i)}} x_{J_{z(i)}}\big] I_{J_{y(i)}}I_{J_{z(i)}}, \\
& =\sum_{\pi \in \mathcal P_2(2k)} \prod_{i=1}^{k} \lim_{n\tends \infty} \E[w^2_p],
 \end{align*}
where  $\pi = \big\{ \{y(1), z(1) \}, \ldots , \{y(k), z(k)  \} \big\}\in \mathcal P_2(2k)$.
Using  Theorem \ref{thm:moments}, from the above last equation, we get
\begin{equation*} 
\beta_{2k}  \geq \sum_{\pi \in \mathcal P_2(2k)} \prod_{i=1}^{k} \beta_2
= \frac{(2k)!}{k! 2^k} (\beta_2)^k.
  \end{equation*}
   Here note that $\beta_2>0$, which can be shown by the similar arguments as used to establish (\ref{eqn:T3_nnormal}). Therefore from the above inequality, $(\beta_{2k})^{\frac{1}{k}} \tends  \infty$. This completes the proof of Proposition \ref{pro:unbdd_supp_Hn}.
\end{proof}

Suppose $\Gamma_p$ is any distribution on $\mathbb{R}$ with moment sequence $\{\beta_k\}$, where $\{\beta_k\}$ is as in (\ref{eqn:lim_moments_wp}).
In the following theorem, we derive the covariance structure for $\{ \Gamma_p : p \text{ is odd postive interger}\}$. Note from Section \ref{subsec:exi+uniq_Gammap} that there could be more than one distribution corresponding to $\{\beta_k\}$. But the covariance structure is independent of the choice of the distributions. The argument used here is similar  to the one used in Theorem \ref{thm:moments} with some technical changes.

\begin{theorem} \label{thm:cov_Hn_odd}
	Let $p_1,p_2$ be odd natural numbers. Then for any input sequence $\{x_i\}$ obeying (\ref{eqn:condition})
	
	\begin{equation}\label{eqn:cov_Hn_odd}
		\lim_{n \rightarrow \infty} \Cov(w_{p_1},w_{p_2})= \frac{1}{2}\sum_{\pi \in {\PP}_2(p_1+p_2)} g_{p_1,p_2}(\pi),
	\end{equation}
where $g_{p_1,p_2}(\pi)$ is as given in Definition \ref{defn:cross-matched int}.
\end{theorem}

\begin{proof}
	  First note from (\ref{eqn:wp_Hn_odd}) and Result \ref{res:trace_Hn} that
	\begin{equation}\label{eqn:cov_limit}
		\lim_{n \rightarrow \infty} \Cov(w_{p_1},w_{p_2})= \lim_{n \rightarrow \infty} \frac{1}{n^{\frac{p_1+p_2}{2}}} \sum_{J_1 \in A_{p_1},J_2 \in A_{p_2}} \left(\E[x_{J_1} x_{J_2}]-\E[x_{J_1}] \E[x_{J_{2}}]\right) I_{J_{1}} I_{J_{2}},
	\end{equation}
	where $A_{p_r}$ is as in (\ref{eqn:Ap_HnOdd}) and $J_r, I_{J_r}, x_{J_r}$ are as in (\ref{eqn:J_I_X}) for $r=1,2$.
	Now observe that the summand in (\ref{eqn:cov_limit}) is non-zero only when 
	\begin{enumerate}[(a)]
		\item each element of $S_{J_1} \cup S_{J_2}$ is repeated at least twice and
		\item there exists at least one cross-matching in $(J_1,J_2)$. 
	\end{enumerate}
	An argument similar to the one employed in the proof of Theorem \ref{thm:moments} implies that only the pair-partitions of $[p_1+p_2]$ contribute in (\ref{eqn:cov_limit}). Furthermore, since $p_1,p_2$ are odd, every pair-partition has at least one cross-matching, and hence $\E[x_{J_1}x_{J_2}]=1$ and $\E[x_{J_1}]=\E[x_{J_2}]=0$ in such cases. As a result, (\ref{eqn:cov_limit}) can be written as
	\begin{equation}\label{eqn:cov exp2}
\lim_{n \rightarrow \infty}\Cov(w_{p_1},w_{p_2})= \lim_{n \rightarrow \infty}\frac{1}{n^{\frac{p_1+p_2}{2}}}\displaystyle \sum_{\pi \in \PP_2(p_1+p_2)}\sum_{j_1,j_2,\ldots , j_{\frac{p_1+p_2}{2}}=-n}^n \sum_{i_1,i_2=1}^n I_{J_1} I_{J_2}\Delta_{J_1}\Delta_{J_2},
	\end{equation}
where $J_1 = (j_1^1,j_2^1,\ldots , j_{p_1}^1)$ with $j_s^1= j_{\pi\left(s\right)}$, $J_2 = (j_1^2,j_2^2,\ldots , j_{p_2}^2)$ with $j_s^2= j_{\pi\left(p_1+s\right)}$ and for $r=1,2$,
 $$I_{J_r}= \prod_{\ell=1}^{p_r} \chi_{\left[\frac{1}{n}, 1\right]}\left(\frac{i_r}{n}-\sum_{q=1}^{\ell}(-1)^q \frac{j_{\pi \left((r-1)p_1+q\right)}}{n}\right), \ \Delta_{J_r}=\delta_{\frac{2i_r}{n}-1-\frac{1}{n}} \left( \sum_{q=1}^{p_r}(-1)^{q} \frac{j_{\pi \left((r-1)p_1+q\right)}}{n}\right).$$

 For a fixed $\pi \in \PP_2(p_1+p_2)$ and $(J_1,J_2)$, $\Delta_{J_i}$ is non-zero only for the solution $i_r \in [n]$ of the following equation 
\begin{equation}\label{eqn:ir, Jr condition}
\sum_{q=1}^{p_r}(-1)^{q} j_{\pi \left((r-1)p_1+q\right)}= 2 i_r-n-1.
\end{equation} 
Consider a combination of $i_r, J_r$ obeying (\ref{eqn:ir, Jr condition}) such that $I_{J_r}=1$. Then each term in the product form of $I_{J_r}$ is equal to 1, which in turn implies that $1 \leq i_r - \sum_{q=1}^{p_r} j_{\pi \left((r-1)p_1+q\right)} \leq n$. Substituting the expression of $i_r$ from (\ref{eqn:ir, Jr condition}), it follows that $i_r$ lies between $1$ and $n$. So now our objective is to find integer solutions of (\ref{eqn:ir, Jr condition}).

Again, by (\ref{eqn:ir, Jr condition}), $i_r$ is an integer if and only if $n$ and $\sum_{q=1}^{p_r} j_{\pi \left((r-1)p_1+q\right)}$ have different parity. Hence the right side of (\ref{eqn:cov exp2}) becomes
\begin{equation}\label{eqn:cov exp 3}
\lim_{n \rightarrow \infty}\frac{1}{n^{\frac{p_1+p_2}{2}}}\displaystyle \sum_{\pi \in \PP_2(p_1+p_2)}\sum_{j_1,j_2,\ldots , j_{\frac{p_1+p_2}{2}}=-n}^n \prod_{r=1}^2 I_{J_r} \ \delta_{-1} \left((-1)^{\sum_{q=1}^{p_r}(-1)^q j_{\pi\left((r-1)p_1+q\right)}+n} \right).
\end{equation}
For $\pi \in \PP_2(p_1+p_2)$, consider the sum $\sum_{q=1}^{p_r}(-1)^{q} j_{\pi \left((r-1)p_1+q\right)}$. For $r=1,2$, we shall use $J_{r,r}$ to denote the self-matching elements in $J_r$ and $J_{1,2}$ to denote the cross-matching elements in $(J_1,J_2)$. Since, the sum $\sum_{j_q^r \in J_{r,r}}(-1)^{q} j_q^r$ is always even for $r=1,2$, the parity of $\sum_{q=1}^{p_r}(-1)^q j_{\pi\left((r-1)p_1+q\right)}$ is same as the parity of $\sum_{j_q^r \in J_{1,2}}(-1)^{q} j_q^r$ for both $r=1,2$.

%
Now, consider $(i_r,J_r)$ with $I_{J_r}>0$ and $\sum_{q=1}^{p_r}(-1)^{q} j_{\pi \left((r-1)p_1+q\right)}=2i_r-n-1$.
Let $j_{k_0}$ be a cross-matched element in $(J_1,J_2)$. Once all other $j_q^r$ are fixed, the parity of $j_{k_0}$ is restricted to either even or odd, depending upon $n$. Also, once all other $j_q^r$ are fixed, the possible values of $j_{k_0}$ such that $I_{J_1}I_{J_2}=1$ falls in an interval. This reduces the number of possibilities for $j_{k_0}$ by half with an error of at most 1. Thus from (\ref{eqn:cov exp 3}),
\begin{align} \label{eqn:limcov_wpq}
\lim_{n \rightarrow \infty}\Cov(w_{p_1},w_{p_2})&= \lim_{n \rightarrow \infty} \Big[\frac{1}{2} \times \frac{1}{n^{\frac{p_1+p_2}{2}}}\displaystyle \sum_{\pi \in \PP_2(p_1+p_2)}\sum_{j_1,j_2,\ldots , j_{\frac{pk}{2}}=-n}^n \prod_{r=1}^2 I_{J_r} + \frac{1}{n^{\frac{p_1+p_2}{2}}} \times O(n^{\frac{p_1+p_2}{2}-1}) \Big] \nonumber \\
&= \frac{1}{2}\sum_{\pi \in \PP_2(p_1+p_2)} g_{p_1,p_2}(\pi),
\end{align}
where $g_{p_1,p_2}(\pi)$ is as given in Definition \ref{defn:cross-matched int}. Note that the last expression is obtained by considering $I_{J_1}I_{J_2}$ as a Riemann sum of the integral $g_{p_1,p_2}(\pi)$.  
\end{proof}

\section{Conclusion}

Our research shows that the behaviour of linear eigenvalue statistics of random Hankel matrices ($w_p$) for odd degree monomials with degree greater than or equal to three is significantly different from the behaviour of linear eigenvalue statistics of random Hankel matrices for even degree monomials and that of linear eigenvalue statistics of Toeplitz matrices.

First for the monomial test functions, we have shown that the moments of linear eigenvalue statistics of Hankel matrix ($\E[w_p]^k$) converge to a limit sequence $\{ \beta_k \}$ as $n$ tends to $\infty$, where each $\beta_k$ is finite. We also showed that there exist probability measures on $\mathbb{R}$ with moment sequence $\{ \beta_k \}$. We proved that any probability measure $\Gamma_p$ with $\{\beta_k \}$ as moment sequence is non-Gaussian and has unbounded support. The behaviour of linear eigenvalue statistics of Hankel matrices for polynomial test functions were discussed in Section \ref{subsec:polytest_ind}.
The simulations in Figure \ref{fig:Limiting_dist} suggest that the linear eigenvalue statistics converge in distribution to a unique limit which is symmetric, unimodular and absolutely continuous. But establishing it theoretically, is difficult, because the moment sequence $\{ \beta_k \}$ might not determine a unique distribution on $\mathbb{R}$. Here the moment sequence $\{ \beta_k \}$ fails to obey Carleman's condition, see Section \ref{subsec:exi+uniq_Gammap}. 
 In this article, we could not conclude whether there is a unique distribution, which corresponds to $\{\beta_k\}_{k\geq 1}$. But we concluded that $w_p$ does not converge in distribution to a Gaussian random variable (Corollary \ref{cor:non-Gaussian_limit}). The question of uniqueness of $\{ \beta_k \}$ along with the convergence in distribution of $w_p$, is still open.



\providecommand{\bysame}{\leavevmode\hbox to3em{\hrulefill}\thinspace}
\providecommand{\MR}{\relax\ifhmode\unskip\space\fi MR }
\providecommand{\MRhref}[2]{%
	\href{http://www.ams.org/mathscinet-getitem?mr=#1}{#2}
}
\providecommand{\href}[2]{#2}

\end{document}